\documentclass[journal]{IEEEtran}
\usepackage{cite}
\usepackage{amsmath,amssymb,amsfonts}
\usepackage{graphicx}
\usepackage{textcomp}
\usepackage{subfigure}
\usepackage{epstopdf}
\usepackage{cases}
\usepackage[ruled,linesnumbered]{algorithm2e}

\usepackage{booktabs}
\usepackage{threeparttable}
\usepackage{multirow}
\usepackage{amsthm}

\newcommand{\tabincell}[2]{\begin{tabular}{@{}#1@{}}#2\end{tabular}}
\newtheorem{theorem}{\textbf{Theorem}}
\newtheorem{lemma}{\textbf{Lemma}}
\newtheorem{assumption}{Assumption}

\newtheorem{remark}{Remark}
\newtheorem{definition}{Definition}

\newtheorem{example}{Example}

\def\BibTeX{{\rm B\kern-.05em{\sc i\kern-.025em b}\kern-.08em
    T\kern-.1667em\lower.7ex\hbox{E}\kern-.125emX}}
\begin{document}
\title{Submodularity-based False Data Injection Attack Scheme in Multi-agent Dynamical Systems }
\author{Xiaoyu Luo$^\dag$, \IEEEmembership{Student Member, IEEE}, Chengcheng Zhao$^\ddag$, Chongrong Fang$^\dag$, \IEEEmembership{Member, IEEE}, \\and Jianping He$^\dag$, \IEEEmembership{Senior Member, IEEE}
	\thanks{$^\dag$: The Department of Automation, Shanghai Jiao Tong University, and Key Laboratory of System Control and Information Processing, Ministry of Education of China, Shanghai 200240, China. E-mail: xyl.sjtu@sjtu.edu.cn, crfang@sjtu.edu.cn, jphe@sjtu.edu.cn. This research work is partially sponsored by NSF of China 61973218 and 62072308.}
	\thanks{$^\ddag$: The State Key Laboratory of Industrial Control Technology and Institute of Cyberspace Research, Zhejiang University, China, and the Department of Electrical and Computer Engineering, University of Victoria, BC, Canada. E-mail: zccsq90@gmail.com. }}


\maketitle

\begin{abstract}
Consensus in multi-agent dynamical systems is prone to be sabotaged by the adversary, which has attracted much attention due to its key role in broad applications. In this paper, we study a new false data injection (FDI) attack design problem, where the adversary with limited capability aims to select a subset of agents and manipulate their local multi-dimensional states to maximize the consensus convergence error. We first formulate the FDI attack design problem as a combinatorial optimization problem and prove it is NP-hard. Then, based on the submodularity optimization theory,
we show the convergence error is a submodular function of the set of the compromised agents, which satisfies the property of diminishing marginal returns. In other words, the benefit of adding an extra agent to the compromised set decreases as that set becomes larger. With this property, we exploit the greedy scheme to find the optimal compromised agent set that can produce the maximum convergence error when adding one extra agent to that set each time. Thus, the FDI attack set selection algorithms are developed to obtain the near-optimal subset of the compromised agents. Furthermore, we derive the analytical suboptimality bounds and the worst-case running time under the proposed algorithms. Extensive simulation results are conducted to show the effectiveness of the proposed algorithm.
\end{abstract}


\section{Introduction}
\label{sec:introduction}
\IEEEPARstart{R}{ecent} years have witnessed the wide application of the network technology to multi-agent dynamical systems where each agent is endowed with its dynamics, including smart grids \cite{zhao2016consensus}, automotive \cite{marzbani2019autonomous}, industrial automation \cite{gungor2009industrial}, etc. Especially, solving consensus problems is of great importance for multi-agent dynamical systems, which aims to reach an agreement via local interaction \cite{olfati2007consensus}.
However, the open network and vulnerable protection provide the chance for the adversary to infiltrate the system secretly to disturb communication links or manipulate the transmitted information, which will ultimately lead to poor control performance, e.g., low convergence rate, huge convergence error, and instability. Meanwhile, the capability of the adversary is usually limited \cite{qin2020optimal}.
Studying the degree of the damage caused by an adversary with limited capacity is beneficial for designing secure algorithms to protect multi-agent dynamical systems.

According to the impact of attacks, we can roughly divide attacks into two categories for multi-agent dynamical systems, including deception attacks (the transmitted information is manipulated) and DoS attacks (the transmitted information is lost or delayed). It has been pointed out by Dibaji \emph{et al.} that deception attacks are the most severe attacks in terms of the level of damages\cite{dibaji2019systems}. Thus, studying deception attacks in multi-agent dynamical systems is important, especially the false data injection (FDI) attack, which is a classical deception attack by injecting false data to manipulate the transmitted information. Before we deploy the multi-agent dynamical systems for practical applications, we need to use host computers to support online reprogramming and debugging via insecure communication protocols for each agent, which allows the adversary to identify cryptographic keys and further reprogram the control rules for multiple agents \cite{Tbone2021}. It is promising to study the scenario that the number of compromised agents is limited because when the number of the compromised agents is too large, it is impossible for the system to maintain stability or reach consensus \cite{dolev1982byzantine}. 
Therefore, to analyze the adversaries' potential impact on consensus thoroughly, it is interesting to explore the scenario where multiple agents are compromised by the adversary with limited capability concurrently.

Lots of literature is dedicated to the multi-agent dynamical systems under FDI attacks against consensus where the number of compromised agents is limited \cite{dolev1986reaching,kieckhafer1994reaching,Leblanc2013Resilient,dibaji2017resilient,fu2019resilient,tsang2020sparse} and the references therein. For example, Dolve \emph{et al.} \cite{dolev1986reaching} and Kiechkhafer \emph{et al.} \cite{kieckhafer1994reaching} constructed an $F$-total FDI attack model where the total number of compromised agents is at most $F$ in network. LeBlance \emph{et al.} \cite{Leblanc2013Resilient} introduced a new notion called network robustness to establish the $F$-local FDI attack model where there are at most $F$ compromised agents among each agent's neighbors. In \cite{tsang2020sparse}, Tsang \emph{et al.} considered an FDI attack in the consensus protocol and formulated a non-convex optimization problem to destabilize the system where at most $F$ agents are compromised at each iteration. The aforementioned works only focus on unidimensional or two-dimensional state interaction and analyze the system performance given some attack strategies.
Furthermore, few works investigate how the adversary with limited capability should select the subset of the compromised agents and mount FDI attacks to cause severe damage. Although there are plenty of works for FDI attack design problems against state estimation \cite{liu2011false,mo2010false2,zhang2019optimal}, the results cannot be applied directly since the problems are different.

Based on the above observations, we are motivated to consider a novel FDI attack design problem where the adversary desires to select a subset of agents and inject false data to their local multi-dimensional states to maximize the convergence error, i.e., the error between the final state with attacks and that without attacks. Our work is oriented toward the design of the attack set selection strategy with limited cost (the total attack cost is limited and known), which is different from the fixed number of compromised agents when the adversary allocates the attack cost to each agent based on its degree in networks \cite{hong2013effect}. Compared to our conference version \cite{2022Luosubmodular}, we extend the analysis of the FDI attack design problem to the case where the deterministic injected false data varies over time, and develop an improved FDI attack set selection algorithm to reduce the running time. Moreover, we provide more details to enrich the motivation, interpretation of submodular conditions of the convergence error, performance analysis, and simulation results.
The main contributions are summarized as follows.

		\begin{itemize}
			\item We investigate the FDI attack design problem where the adversary with limited capability aims to maximize the convergence error by compromising a subset of agents and manipulating their local multi-dimensional states. 
			\item We prove the considered problem is NP-hard and show that the convergence error is a submodular function of the set of the compromised agents.
			\item We develop two FDI attack set selection algorithms to obtain the near-optimal compromised subset, combining the submodular conditions and greedy algorithm. The proposed algorithms are applicable to all kinds of deterministic attack strategies.
			\item We derive the analytical suboptimality bounds and the worst-case running time under the proposed algorithms. 
			Extensive simulation results are conducted to show the effectiveness of the proposed algorithms.
		\end{itemize}
		
		The rest of the paper is organized as follows. Section \ref{II} introduces the system model and the attack model, and formulates the FDI attack design problem. In Section \ref{III}, the submodularity of the convergence error is analyzed and an algorithm is designed to obtain the suboptimality bound. In Section \ref{IV}, we extend the results to the case when the injected false data varies over time. In addition, further discussion is shown in Section \ref{V}. Simulation results are presented in Section \ref{VI}. Finally, we conclude our work in Section \ref{VII}.\\
\textbf{Notations.} Let $\mathbb{R}_{\geq 0}$, $\mathbb{C}$, and $\mathbb{Z}^+$ denote the set of non-negative real numbers, the set of complex numbers, and the set of positive integers, respectively. For a vector $\mathbf{p} \in \mathbb{R}^{n}$, we let $\|\mathbf{p}\|$ denote its $l_2$-norm, and $\mathbf{p}^\mathrm{T}$ denote its transpose. We denote $I_n$ and $\textbf{1}_n$ as the $n$ dimensional diagonal unit matrix and column vector, respectively. We denote $e_i$ as the canonical vector with $1$ in the $i$-th entry and $0$ elsewhere. For matrix $P\in \mathbb{R}^{n\times n}$, we use $\mathrm{Ker}(P)$ to denote its
kernel space, and $\mathrm{rank}(P)$ to denote its rank. The symbol $\otimes$ denotes the Kronecker product and $\mathrm{diag}(\cdot)$ denotes the diagonal matrix. The set of $0$-$1$ indicator vectors of dimension $n$ is denoted as $\{0,1\}^{n}$.

\section{PROBLEM FORMULATION}\label{II}	
\subsection{Network Model}
Consider a system composed by $n\in\mathbb{Z}^+$ homogeneous autonomous agents and each agent has a unique ID number denoted by $i={1,2,\cdots,n}$. A strongly undirected connected graph $\mathcal{G}=\{ \mathcal{V}, \mathcal{E} \}$ is used to model the communication topology among $n$ agents, where $\mathcal{V}=\{1, 2, \cdots, n \}$ is the set of agents and $\mathcal{E}\subseteq \mathcal{V} \times \mathcal{V}$ is the edge set. The neighbor set of agent $i$ is denoted by ${\mathcal{N}_i} = \left\{ {j|\left( {i,j} \right) \in \mathcal{E}},\forall j \in \mathcal{V} \right\}$ with its cardinality $d_i=|\mathcal{N}_i|$, where $(i, j)\in \mathcal{E} $ illustrates that agent $i$ can receive information from agent $j$. The degree matrix is a diagonal matrix defined as $D=\text{diag}\{D_{ij}\} \in \mathbb{R}^{n\times n}$ with $D_{ii} =d_i$ for all $i\in \mathcal{V}$. Additionally, we define $d_{\mathrm{max}}$ as the maximum degree over all agents in the graph. The adjacency matrix is represented by $V=\left[ v_{ij}\right]\in \mathbb{R}^{n\times n}$, with $v_{ij}=1$ if $(i, j)\in \mathcal{E} $, and $v_{ij}=0$ otherwise. We do not consider self-loops here, i.e., $v_{ii}=0$.
Then, the Laplacian matrix is written as $L=D-V$. 

\subsection{System Dynamic Model }
Each agent has a continuous time-invariant linear system model and the dynamics of the $i$-th agent for any $i \in \mathcal{V} $ are
\begin{equation}\label{1}
\dot{\mathbf{x}}_i(t)=A\mathbf{x}_i(t)+B \mathbf{u}_i(t),  
\end{equation}
where $A$, $B$ are constant matrices, $\mathbf{x}_i(t)\in \mathbb{R}^{m}$ and $\mathbf{u}_i(t)\in \mathbb{R}^{m}$ with $m\in \mathbb{Z}^{+}$ denote the state and the control input of agent $i$ at time $t$, respectively. We use the following linear consensus protocol \cite{Olfati2004Murray}, i.e., 
\begin{equation}\label{2}
\mathbf{u}_i(t)=\sum_{j\in \mathcal{N}_i}W_{ij}(\mathbf{x}_j(t)-\mathbf{x}_i(t)),
\end{equation} 
where $W\in \mathbb{R}^{m\times m}$ is a weight matrix with nonnegative entries. The entries are chosen as $W_{ij}=v_{ij} \bar{d}$ where $\bar{d} \in (0,\frac{1}{d_{\mathrm{max}}})$ \cite{xiao2004fast}. 
Then, for any $i \in \mathcal{V}$, 
\begin{equation}\label{3}
\dot{\mathbf{x}}_i(t)= A \mathbf{x}_i(t)+B \sum_{j\in \mathcal{N}_i} v_{ij}\bar{d} (\mathbf{x}_j(t)-\mathbf{x}_i(t)).
\end{equation}
Let $\mathbf{x}(t)=[\mathbf{x}_1(t)^{\mathrm{T}},\cdots,\mathbf{x}_n(t)^{\mathrm{T}}]^{\mathrm{T}}\in \mathbb{R}^{mn}$ be the global state of the system, whose dynamics can be written in a compact form as follows.
\begin{equation}\label{4}
\dot{\mathbf{x}}(t)=M \mathbf{x}(t), 
\end{equation}
where $M=(I_n \otimes A-\bar{d}L\otimes B) \in \mathbb{R}^{mn \times mn} $ is the system matrix and its eigenvalues are denoted as $\sigma_i$ with $ i\in \{1,2,\cdots, mn\}$. The spectrum of matrix $L$ is denoted by ${\lambda_1,\lambda_2,\cdots,\lambda_{n}}$ and satisfies $\mathrm{Re(\lambda_1)}\leq \mathrm{Re(\lambda_2)}\leq \cdots\leq \mathrm{Re(\lambda_n)}$, for any $i\in \mathcal{V}$. Then, there exists a non-singular matrix $U\in \mathbb{C}^{n\times n}$ such that $ULU^{-1}=J$, where $J$ is the diagonal matrix whose diagonal elements are eigenvalues $\lambda_i, i\in \mathcal{V}$, and the columns of the complex matrix $U$ are the corresponding eigenvectors of $L$. Here, we provide the following assumption and lemma on the global system stability. 

\begin{assumption}\label{ass1}
	The dynamic system $(A,B)$ is controllable. The global system (\ref{4}) can be Hurwitz stable under the chosen control protocol (\ref{2}).
\end{assumption}

\begin{lemma}\label{lemma1}
		(\textbf{Consensus conditions\cite{xiao2007consensus})}
	System (\ref{4}) solves a consensus problem if and only if 
	\begin{itemize}
		\item [i)] $\mathrm{rank}(M)=\mathrm{rank}(M^2)$;
		\item [ii)] For $\forall i \in \mathcal{V}$, it holds $\sigma_i=0$ or $\mathrm{Re(\sigma_i)}<0$; 
		\item [iii)] if $\sigma_i=0$, $\exists i \in \mathcal{V}$, then for any $\xi \in \mathrm{Ker}(M)$, there exists a vector $\mathbf{b}\in \mathbb{R}^{m}$ such that $\xi=\mathbf{1}_{n}\otimes \mathbf{b}$.
	\end{itemize}
\end{lemma}
\begin{remark}
	Under Assumption \ref{ass1}, if a system is not controllable, the consensus protocol (\ref{2}) cannot control the initial states $\mathbf{x}(0)$ to any final states $\mathbf{x}(t)$ in a period. Therefore, the controllability of systems is necessary to guarantee the effectiveness of the consensus protocol. In addition, the system is Hurwitz stable, i.e., all eigenvalues have the negative real part, which satisfies the consensus conditions in Lemma \ref{lemma1}.
\end{remark}

\subsection{Adversary Model}
In this part, we explore the impact of multiple compromised agents on system performance with limited cost. Consider the scenario that the adversary desires to compromise a subset of agents and injects false data in the update process of the agents to maximize convergence error, shown in Fig. \ref{attack_scenario}. The adversary needs certain cost to compromise the agent, i.e., attack cost $c(i)\in \mathbb{R}_{\geq 0}$ for each agent $i\in \mathcal{V}$. In addition, the attack cost can be quite different in many real-world networks and the performance of attack strategies is sensitive to the total attack cost \cite{hong2013effect}. 
\begin{figure}[t]
	\centering
	\includegraphics[width=0.45\textwidth]{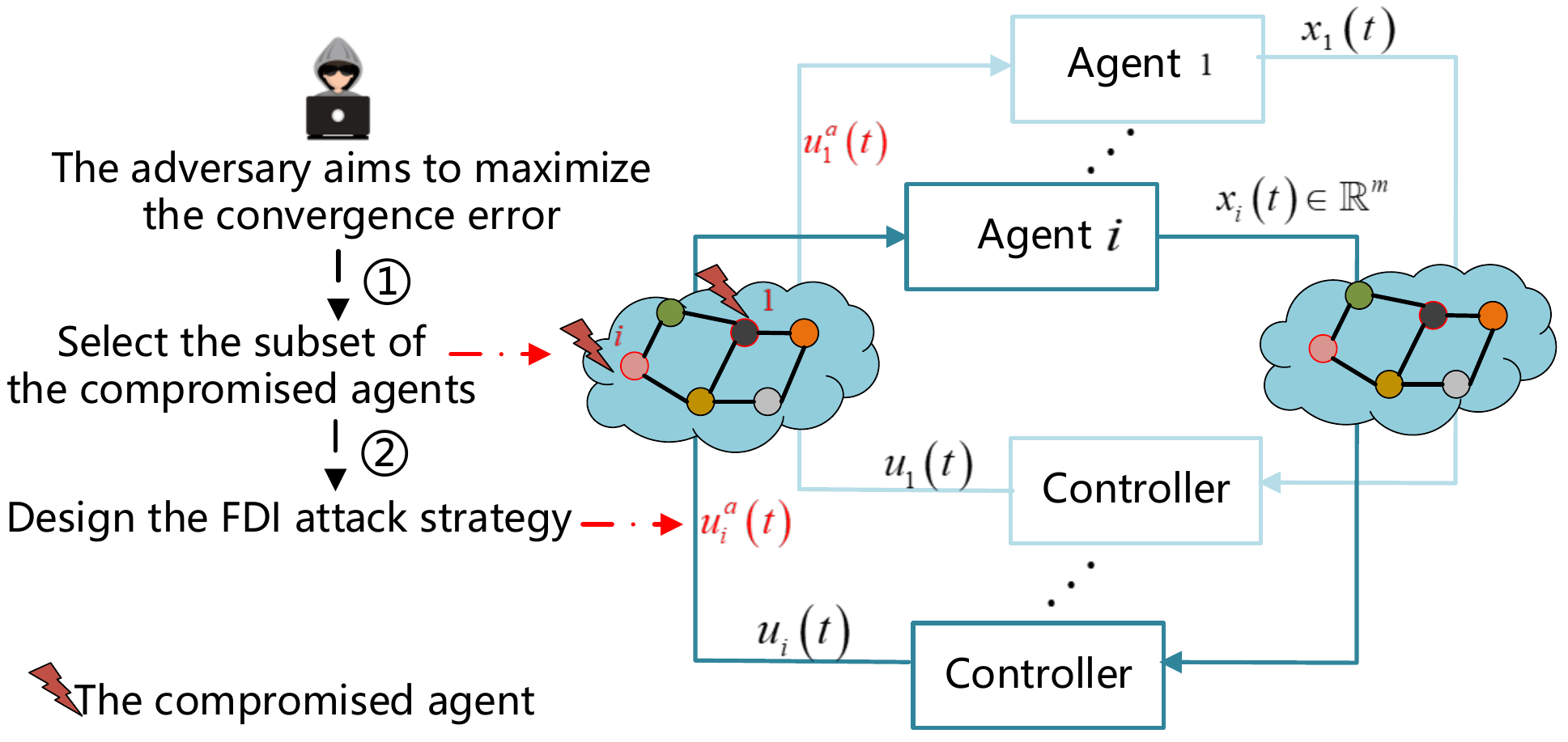}
	\caption{The FDI attack scenario in continuous-time system }
	\label{attack_scenario}
	\vspace*{-10pt}
\end{figure}
Next, we divide the attack process into the following two parts. \\
$\bullet$ \textbf{The selection of the subset of  the compromised agents}\\
We define $\mathcal{A}$ as the subset of the compromised agents. If agent $i$ is compromised, then $i\in \mathcal{A}\subseteq \mathcal{V}$. We define a vector  $\mathbf{\mu}^{\mathcal{A}} \in\{0,1\}^{n}$ to be the indicator vector of the compromised agents, i.e., $\mathbf{\mu}^{\mathcal{A}}=[\mu_1^{\mathcal{A}},\cdots,\mu_n^{\mathcal{A}}]^{\mathrm{T}}$.
The element $\mu_i^{\mathcal{A}} =1$ if and only if agent $i$ is compromised, i.e., 
\begin{numcases}{\mu_i^{\mathcal{A}}:=}\label{8}
\begin{aligned}
&1, ~\mathrm{if} ~i\in \mathcal{A},\\
&0, ~\mathrm{if} ~i\notin \mathcal{A}.
\end{aligned}
\end{numcases}
$\bullet$ \textbf{The design of attack strategy of the FDI attack}\\
For any selected subset $\mathcal{A}$ of the compromised agents, agent $i$ will update its state as
\begin{equation}\label{5}
\dot{\mathbf{x}}_i(t)=A \mathbf{x}_i(t)+ B \mathbf{u}_i(t)+ \mathbf{u}_i^{a}(t), ~i\in \mathcal{V},
\end{equation}
where $\mathbf{u}_i^{a}(t)$ is the false data injected by the adversary. Since the agents are compromised by an adversary at the same time,  
		the false data is designed as follows.
		\begin{numcases}{\mathbf{u}_i^{a}(t)=}\label{6}
		\begin{aligned}
		&\mathbf{\theta}(t), ~\forall i \in \mathcal{A},\\
		&\mathbf{0}, ~~~~\forall i \notin \mathcal{A},
		\end{aligned}
		\end{numcases}
		and $\mathbf{\theta}(t)\in \mathbb{R}^{m}$ is the concrete attack strategy. Additionally, the dynamical interaction between agents with $m$-dimensional states has greatly expanded the attack space.
		
		Then, we make the following assumptions on the ability of the adversary.
		\begin{assumption}\label{ass2}
			The total attack cost of the adversary has an upper bound $\Omega \in \mathbb{R}_{\geq 0}$, i.e., $\sum_{i\in \mathcal{A}} c(i) \leq \Omega$.
		\end{assumption}
		
		\begin{assumption}\label{ass4}
			The adversary can inject false data without being detected, and the false data is bounded.
		\end{assumption}
		
		\begin{assumption}\label{ass5}
			The attack strategy of the adversary can be divided into two categories, including time-invariant and time-variant attack strategy. Both attack strategies are deterministic.
		\end{assumption}
		\begin{remark}
			Assumption \ref{ass2} gives a reasonable limit on the capability of the adversary.
			In almost all types of attacks, energy cost is inherent \cite{zhang2015optimal}, which will influence the attack strategy. In addition to the energy cost, compromising an agent needs other attack cost and the adversary cannot compromise agents arbitrarily. For example, in smart grids, some meters are under physical protection and the adversary is constrained to some specific meters \cite{liu2011false}. In Assumption \ref{ass4}, when the adversary enables access to the system environment and obtains the configuration information of dynamical systems, it is not difficult to inject false data, like Stuxnet and Ukraine attacks. Especially, motivated by the work \cite{guo2018worst} where
			the Kullback-Leibler divergence as a stealthiness metric requires no more than a threshold, it is reasonable that the injected false data is bounded. 
			In Assumption \ref{ass5}, the time-invariant attack strategy satisfies $\mathbf{\theta}(t)=\mathbf{\theta}$ where the injected false data $\mathbf{\theta}$ is constant and does not vary with the time. Instead, under time-variant attack strategy, the injected false data $\mathbf{\theta}(t)$ is continuous integrable and varies over the time. 
\end{remark}

\subsection{Problem Formulation}
The dynamics of the compromised agents are
\begin{equation}\label{9}
\dot{\mathbf{x}}(\mathbf{\mu}^{\mathcal{A}},t)=M\mathbf{x}(t)+\mathbf{\mu}^{\mathcal{A}}\otimes \mathbf{\theta}(t).
\end{equation}
Then, we derive 
\begin{equation}\label{10}
\mathbf{x}(\mathbf{\mu}^{\mathcal{A}},t)=e^{Mt}\mathbf{x}(0)+\int_{0}^{t}e^{M(t-\tau)}\mathbf{\mu}^{\mathcal{A}}\otimes \mathbf{\theta}(\tau)d\tau.
\end{equation}

Our objective is to select the compromised subset $\mathcal{A}$ such that the convergence error is maximized while the limited total attack cost is considered. We define
\begin{align}\label{add10}
f(\mathcal{A}) \triangleq \lim_{t\to t_c}\|\mathbf{x}(\mathbf{\mu}^{\mathcal{A}},t)-\mathbf{x}(t)\|
\end{align}
as the convergence error, and quantify the effects of the attack set selection strategy on convergence error in a period, where $t_c$ is greater than the convergence time without attacks since the states after that time will converge to a common value.
Then, the FDI attack design problem is formulated as follows. 
\begin{align}\label{11}
\mathbf{\mathcal{P}_1}: \quad &\mathrm{max}_{\mathcal{A}\subseteq \mathcal{V}} ~f(\mathcal{A})\\
&\mathrm{s.t.}~  \|\mathbf{u}_i^{a}\|\leq \bar{u},  ~\forall i\in \mathcal{V},\nonumber\\
&\quad ~~ \sum_{i\in \mathcal{A}}c(i) \leq \Omega, \nonumber\\
&\quad ~ \int_{0}^{t_c}\mathbf{\theta}(\tau) d\tau\leq \bar{g}, ~\forall i\in \mathcal{V},\nonumber
\end{align}
where $\bar{u}$ is the upper bound of false data, and $\bar{g}$ is the upper bound of attack energy. The problem $\mathcal{P}_1$ is a combinatorial optimization problem.

\section{Time-invariant attack set selection}\label{III}
In this section, we investigate the case with time-invariant attacks by using the submodularity optimization theory [see Section III.3. in \cite{wolsey1999integer}]. First, we show the problem $\mathcal{P}_1$ is NP-hard. Then, we derive the conditions that the convergence error in problem $\mathcal{P}_1$ is submodular. Finally, we design an FDI-ASSA to solve it and show the analytical gap between the solution and the optimal one under FDI-ASSA.

\subsection{NP-hardness of Problem $\mathcal{P}_1$}
First, we show the complexity of the problem $\mathcal{P}_1$ under time-invariant attacks.	
\begin{lemma}\label{NPhard}
	\textbf{(Complexity)}.
	The problem $\mathcal{P}_1$ under the time-invariant attack strategy $\mathbf{\theta}(t)=\mathbf{\theta}$ is NP-hard.
\end{lemma}	
\begin{proof}
	We give a reduction from the problem $\mathcal{P}_1$ to the $0$-$1$ knapsack problem \cite{pisinger1998knapsack}. For the ease of notion, in this part we drop the subscript $mn \times mn$ from $I_{mn \times mn}$. 
	
	First, we find the reduction of problem $\mathcal{P}_1$ under time-invariant attacks. Let $\kappa(\mathcal{A},t)=\mathbf{x}(\mathbf{\mu}^{\mathcal{A}},t)-\mathbf{x}(t)$. When $M=aI$ with $a<0$ being some constant and $\mathbf{\theta}=[\theta_1, \cdots, \theta_m]^{\mathrm{T}}$ with $\theta_i \in \mathbb{R}_{\geq 0}$, we know 
	\begin{align*}
	\kappa(\mathcal{A},t)=-(\frac{1}{a}I -\frac{e^{aIt}}{aI})(\mathbf{\mu}^{\mathcal{A}}\otimes \mathbf{\theta}).
	\end{align*}
	Then, it follows that its $i$-th segmented vector
	\begin{align*}
	\kappa(\mathcal{A},t)_i =\mathrm{diag}\{-\frac{1}{a}-\frac{e^{at}}{a}\}_{m\times m} \mu_i^{\mathcal{A}} \mathbf{\theta} 
	\end{align*}
	if $\mu_i^{\mathcal{A}}=1$ and $\kappa(\mathcal{A},t)_i =0$ if $\mu_i^{\mathcal{A}}=0$. Since
	{\small{
	\begin{align*}
	\|\kappa(\mathcal{A},t)_i\|(\mu_i^{\mathcal{A}}=1)-\|\kappa(\mathcal{A},t)_i\|(\mu_i^{\mathcal{A}}=0)= \|\kappa(\mathcal{A},t)_i\|,
	\end{align*}}}the reduction of problem $\mathcal{P}_1$ by adding attack agent $i$ is $\|\kappa(\mathcal{A},t)_i\|$ for any $i$.

	Then, we find the corresponding instance of problem $\mathcal{P}_1$ for $0$-$1$ knapsack problem.
	Given the number of items $n$, the set of values $\{\alpha_i\}$, the set of weights $\{\beta_i\}$ and the weight budget $\beta$ for the $0$-$1$ knapsack problem, the corresponding instance of problem $\mathcal{P}_1$ under time-invariant attacks is the stable system matrix $M=-I$ and the attack strategy $\mathbf{\theta}=[\theta_1, \cdots, \theta_m]^{\mathrm{T}}$ with $ \|\theta\| =(-\frac{a}{1+e^{at}}) \alpha_i=(\frac{1}{1+e^{-t}}) \alpha_i $, based on $\|\kappa(\mathcal{A},t)_i\|= (-\frac{1}{a}-\frac{e^{at}}{a}) \|\theta\|$. Then, we can see that the selected subset $\mathcal{A}$ for $0$-$1$ knapsack problem is optimal if and only if it is optimal for the corresponding problem $\mathcal{P}_1$. Since the optimization form of  $0$-$1$ knapsack problem is NP-hard, the problem $\mathcal{P}_1$ is NP-hard.
\end{proof}

\begin{remark}
	In Lemma \ref{NPhard}, we reduce the proposed problem $\mathcal{P}_1$ under time-invariant attacks to the $0$-$1$ knapsack problem that has been shown to be NP-hard, where the proof is similar to Lemma $2$ in \cite{zhang2017sensor}. Lemma \ref{NPhard} reveals that it is hard to find an optimal solution for problem $\mathcal{P}_1$ in polynomial time.
\end{remark}

\subsection{Submodular Conditions under Time-invariant Attacks}
We introduce the following definitions, which are basic and important in the submodularity optimization approach. Recall that $\mathcal{V}$ is the set of all agents. A set function $f$ over $\mathcal{V}$ assigns a real value to every subset of $\mathcal{V}$, i.e., $f:2^{\mathcal{V}} \rightarrow \mathbb{R}$.
\begin{definition}
	\textbf{(Normalized and Monotone \cite{jawaid2015submodularity})}. The function $f$ is normalized if $f({\emptyset})=0$. The function $f$ is monotone non-decreasing if for all $\mathcal{A}\subseteq \mathcal{B}\subseteq \mathcal{V}$, it holds $f(\mathcal{A})\leq f(\mathcal{B})$. 
\end{definition}
\begin{definition}
	\textbf{(Submodularity \cite{jawaid2015submodularity})}. The function $f$ is submodular if all $\mathcal{A}\subseteq \mathcal{B} \subseteq \mathcal{V}$ and for all $j\in \mathcal{V} \backslash \mathcal{B}$, it holds $f(\mathcal{A} \cup \{j\})-f(\mathcal{A})\geq f(\mathcal{B} \cup \{j\})-f(\mathcal{B})$. 
\end{definition}

The submodular function has the property of diminishing marginal returns \cite{jawaid2015submodularity}. It means that the benefit of adding a new element to the total set decreases as the total set becomes larger. When the benefit is close to zero, the set reaches the saturation point. 

With the property of the submodular function, we transform the convergence error in (\ref{11}) and give its equivalent expression as follows.

\begin{lemma}\label{objective}
	\textbf{(Equivalence)}.
	Under (\ref{8}) with time-invariant attack strategy $\theta$, the convergence error of system (\ref{9}) satisfies
	\begin{equation}\label{add13}
	f(\mathcal{A})=\|(U^{\mathrm{T} }\otimes I_m) \hat{M} (U \mathbf{\mu}^\mathcal{A} \otimes \mathbf{\theta})\|
	\end{equation}
	where 
	\begin{align*}
	\hat{M}=\mathrm{diag}\{\varphi(A),\varphi(A-\lambda_2 \bar{d} B), \cdots, \varphi(A-\lambda_n \bar{d} B)\}
	\end{align*}with $\varphi(A) \triangleq \int_{0}^{t_c}e^{A(t-\tau)}\text{d}\tau$.
\end{lemma}
\begin{proof}
	Let $\mathbf{z}(\mathbf{\mu}^{\mathcal{A}},t)=(U\otimes I_m)\mathbf{x}(\mathbf{\mu}^{\mathcal{A}},t)$. Combining it with (\ref{9}), we have   
	\begin{align*}
	\dot{\mathbf{z}}(\mathbf{\mu}^{\mathcal{A}},t)=&(U\otimes I_m)(I_n\otimes A-\bar{d} L\otimes B)(U^{\mathrm{T}}\otimes I_m)\mathbf{z}(\mathbf{\mu}^{\mathcal{A}},t) \\
	&+(U\otimes I_m)(\mathbf{\mu}^{\mathcal{A}}\otimes \mathbf{\mathbf{\theta}}(t)) \nonumber\\
	=&(U I_n U^{\mathrm{T}} \otimes I_m A I_m)\\
	&-(U \bar{d}L U^{\mathrm{T}} \otimes I_m B I_m) \mathbf{z}(\mathbf{\mu}^{\mathcal{A}},t)\\
	&+(U \mathbf{\mu}^{\mathcal{A}}\otimes I_m \mathbf{\mathbf{\theta}}(t)) \nonumber \\
	=& (I_n \otimes A - \bar{d} J \otimes B)\mathbf{z}(\mathbf{\mu}^{\mathcal{A}},t)\\
	&+ (\sum_{i\in \mathcal{A}}U_i \otimes \mathbf{\mathbf{\theta}}(t)),
	\end{align*}
	where $L=U J U^{\mathrm{T}}$ with $U=[U_1,U_2,\cdots,U_n]\in \mathbb{R}^{n \times n}$.
	
	Next, recall that the consensus condition without attacks in Lemma \ref{lemma1}. We assume that only the $n-1$ matrices $A-\lambda_i \bar{d} B, i\in \{2,\cdots,n\}$ are Hurwitz and there exists a vector-valued function $\mathbf{r}: \mathbb{R}_{\geq 0} \rightarrow \mathbb{R}^{m}$ that is exponentially vanishing as $t\rightarrow t_c$ and such that 
	\begin{align*}
	\mathbf{z}(\mathbf{\mu}^{\mathcal{A}},t_c)
	=&((\mathbf{e}_1 \mathbf{e}_1^{\mathrm{T}})\otimes e^{At_c})\mathbf{z}(0)+\mathbf{r}(t_c) 
	+\hat{M} (U \mathbf{\mu}^\mathcal{A} \otimes \mathbf{\mathbf{\mathbf{\theta}}})
	\end{align*} with
	\begin{align*}
	\hat{M}=\mathrm{diag}\{\varphi(A),\varphi(A-\lambda_2 \bar{d} B), \cdots, \varphi(A-\lambda_n \bar{d} B)\}.
	\end{align*}
	Then, let $\kappa(\mathcal{A},t)=\mathbf{x}(\mathbf{\mu}^{\mathcal{A}},t)-\mathbf{x}(t)$ and $\mathbf{x}(\mathbf{\mu}^{\mathcal{A}},t)$ satisfies   
	\begin{align*}
	\lim_{t\rightarrow t_c} \mathbf{x}(\mathbf{\mu}^{\mathcal{A}},t)=&(U^{\mathrm{T}}\otimes I_m) \lim_{t\rightarrow t_c} \mathbf{z}(\mathbf{\mu}^{\mathcal{A}},t)\\
	=& (U^{\mathrm{T} }\otimes I_m)[(\mathbf{e_1} \mathbf{e_1}^{\mathrm{T}}) \otimes e^{At_c}] (U \otimes I_m) \mathbf{x}(0)\\
	& + \lim_{t\rightarrow t_c} \kappa(\mathcal{A},t).
	\end{align*}
	where
	\begin{align*}
	\lim_{t\rightarrow t_c} \kappa(\mathcal{A},t)=(U^{\mathrm{T} }\otimes I_m) \hat{M} (U \mathbf{\mu}^\mathcal{A} \otimes \mathbf{\mathbf{\theta}}).
	\end{align*} 
	Therefore, the proof is completed.
\end{proof}

Lemma \ref{objective} provides the equivalent convergence error expression (\ref{add13}), which is beneficial for analyzing the submodular condition later. Note that $\|(U^{\mathrm{T} }\otimes I_m) \hat{M} (U \mathbf{\mu}^\mathcal{A} \otimes \mathbf{\mathbf{\theta}})\|$ is the approximation of convergence error $f(\mathcal{A})$ since $\mathbf{r}(t_c)$ is not equivalent to $\mathbf{0}$ when $t\to t_c$. For each agent, due to the same tiny deviation between the elements in $\mathbf{r}(t_c)$ and $0$, the deviation will not influence the submodularity of the convergence error. Thus, 
we can replace $f(\mathcal{A})$ with $\|(U^{\mathrm{T} }\otimes I_m) \hat{M} (U \mathbf{\mu}^\mathcal{A} \otimes \mathbf{\mathbf{\theta}})\|$ in some extent. As shown in the following theorem, we provide the necessary and sufficient conditions to guarantee the submodularity of (\ref{add13}).

\begin{theorem}\label{firstsubmodular}
		\textbf{(Submodular conditions)}.
	Under Assumptions $1$-$3$ and time-invariant attack strategy $\mathbf{\theta}$, the convergence error $f(\mathcal{A})$ in problem $\mathcal{P}_1$ is normalized and monotone non-decreasing submodular iff
	\begin{align}\label{13}
	h(\mathcal{A} \cup \mathcal{B},\mathcal{B} \backslash \mathcal{A}) \geq 0
	\end{align}
	and 
	\begin{align}\label{add14}
	h(\mathcal{A},\{j\}) \geq \frac{1}{2}(\gamma-1) h(\{j\},\{j\}) + \gamma h(\mathcal{B},\{j\})
	\end{align}
	always hold where 
	\begin{align*}
	h(\mathcal{A},\mathcal{B}) \triangleq [(\sum_{i\in \mathcal{A}} U_i)^{\mathrm{T}} \otimes \mathbf{\theta}^{\mathrm{T}} ] \hat{M}^{\mathrm{T}} \hat{M} (\sum_{i \in \mathcal{B}} U_i \otimes \mathbf{\theta})
	\end{align*}
	with all $\mathcal{A} \subseteq \mathcal{B \subseteq \mathcal{V}}$, and $\gamma \in [0,1]$ is determined by the selected sets $\mathcal{A}$ and $\mathcal{B}$.
\end{theorem}
\begin{proof}
	The proof is divided into two parts. The first part \textit{(If)} is to prove the sufficiency of the submodular conditions. The second part \textit{(Only if)} is to show the submodular conditions' necessity.
	
	\textit{(If)} When (\ref{13}) is established, we have
	\begin{align*}
	h(\mathcal{B} \backslash \mathcal{A},\mathcal{B} \backslash \mathcal{A}) + 2 h(\mathcal{A},\mathcal{B} \backslash \mathcal{A}) \geq 0.
	\end{align*}
	Since
	{\small{
			\begin{align*}
			&f(\mathcal{A})=\{(U^{\mathrm{T} }\otimes I_m) \hat{M} (U \mathbf{\mu}^\mathcal{A} \otimes \mathbf{\mathbf{\theta}})^{\mathrm{T}} (U^{\mathrm{T} }\otimes I_m) \hat{M} (U \mathbf{\mu}^\mathcal{A} \otimes \mathbf{\mathbf{\theta}})\}^{\frac{1}{2}} \nonumber\\
			&= \{[(\sum_{i\in \mathcal{A}} U_i)^{\mathrm{T}} \otimes \mathbf{\mathbf{\mathbf{\theta}}}^{\mathrm{T}} ] \hat{M}^{\mathrm{T}} \hat{M} (\sum_{i \in \mathcal{A}} U_i \otimes \mathbf{\mathbf{\mathbf{\theta}}})\}^{\frac{1}{2}}\\
			&= h(\mathcal{A},\mathcal{A})^{\frac{1}{2}},
			\end{align*}}}
	and
	\begin{align*}
	&f(\mathcal{B})=h(\mathcal{B},\mathcal{B})^{\frac{1}{2}} \nonumber\\
	&=\{h(\mathcal{A},\mathcal{A})+h(\mathcal{B} \backslash \mathcal{A},\mathcal{B} \backslash \mathcal{A}) + 2 h(\mathcal{A},\mathcal{B} \backslash \mathcal{A})\}^{\frac{1}{2}},
	\end{align*} 
	it follows that $f(\mathcal{A}) \geq f(\mathcal{B})$. Hence, the monotonicity of the convergence error satisfies when (\ref{13}) holds. In addition, if $\mathcal{A}={\emptyset}$, we have $\mathbf{\mu}^{\mathcal{A}}=\mathbf{0}$. Thus, it is inferred that $f({\emptyset})=0$ and the convergence error is normalized.
	
	For the adding element $j\in \{a_{t+1},\cdots,a_{n}\}$, which satisfies $j\in\mathcal{V} \backslash \mathcal{B}$, the requirement of submodularity is
	\begin{align*}
	f(\mathcal{A} \cup \{j\})-f(\mathcal{A})\geq f(\mathcal{B} \cup \{j\})-f(\mathcal{B}).
	\end{align*} 
	This can be rewritten as
	{\small{	
	\begin{align}\label{14}
	\{h(\mathcal{A},\mathcal{A})+ h(\{j\},\{j\}) + 2h(\mathcal{A},\{j\})\}^{\frac{1}{2}} - h(\mathcal{A},\mathcal{A})^{\frac{1}{2}} \geq \nonumber \\
	\{h(\mathcal{B},\mathcal{B})+ h(\{j\},\{j\}) + 2h(\mathcal{B},\{j\})\}^{\frac{1}{2}} - h(\mathcal{B},\mathcal{B})^{\frac{1}{2}}.
	\end{align}}}
	Let $\rho_j(\mathcal{A}) \triangleq f(\mathcal{A} \cup \{j\})-f(\mathcal{A})$ and $\rho_j^{+}(\mathcal{A}) \triangleq f(\mathcal{A} \cup \{j\})+f(\mathcal{A})$. Taking the left hand side of the above inequality (\ref{14}) into consideration, we obtain 
	\begin{align*}
	\rho_j(\mathcal{A})=(\rho_j^{+}(\mathcal{A}))^{-1} (h(\{j\},\{j\}) + 2h(\mathcal{A},\{j\})).
	\end{align*}
	Similarly, for the right hand side of (\ref{14}), it is easy to infer
	\begin{align*}
	\rho_j(\mathcal{B})=(\rho_j^{+}(\mathcal{B}))^{-1} (h(\{j\},\{j\}) + 2h(\mathcal{B},\{j\})).
	\end{align*}		
	When (\ref{add14}) is established, we have
	\begin{align*}
	\rho_j(\mathcal{A}) \geq& (\rho_j^{+}(\mathcal{A}))^{-1} \gamma (h(\{j\},\{j\}) + 2h(\mathcal{B},\{j\})),
	\end{align*}
	for any $j\in \mathcal{V} \backslash \mathcal{B}$. Owing to the monotonicity of the convergence error, it follows that $\rho_j^{+}(\mathcal{B}) \geq \rho_j^{+}(\mathcal{A}) \gg 1$. We define
	\begin{align*}
	\gamma=\frac{(\rho_j^{+}(\mathcal{B}))^{-1}}{(\rho_j^{+}(\mathcal{A}))^{-1}},
	\end{align*}
	where $\gamma \in [0,1]$. Then, it is easy to obtain 
	\begin{align*}
	\rho_j(\mathcal{A}) \geq \rho_j(\mathcal{B})
	\end{align*}
	Thus, the submodularity is satisfied.
	
	\textit{(Only if)} 
	Suppose that (\ref{13}) and (\ref{add14}) are not established. Then we have 
	\begin{align*}
	h(\mathcal{B} \backslash \mathcal{A},\mathcal{B} \backslash \mathcal{A}) + 2 h(\mathcal{A},\mathcal{B} \backslash \mathcal{A}) \leq 0
	\end{align*}
	and
	\begin{align*}
	\rho_j(\mathcal{A}) \leq (\rho_j^{+}(\mathcal{A}))^{-1} \gamma (h(\{j\},\{j\}) + 2h(\mathcal{B},\{j\})). 
	\end{align*}
	Thus, $f(\mathcal{B}) \leq f(\mathcal{A})$ and $\rho_j(\mathcal{B}) \geq \rho_j(\mathcal{A})$ are established, which do not meet the monitonicity and submodularity.  
\end{proof}

Based on Theorem \ref{firstsubmodular}, we know the submodular conditions are associated with the attack strategy $\mathbf{\theta}$, Laplacian matrix $L$, and system dynamics matrices. Meanwhile, Theorem \ref{firstsubmodular} reveals the diminishing properties of the convergence error and provides the basis for solving the combinatorial optimization problem $\mathcal{P}_1$ under the time-invariant attack strategy. 

To explain the submodular conditions intuitively, we take the following example to show the process of computing the submodular conditions.
\begin{example}
	Consider a system with six agents connected in a  linear graph and the dynamics of each agent satisfy $A=[-0.5~0;1~-1]$ and $B=[0.1~0.1;0.5~ 0.2]$. The injected $\mathbf{\theta}=[0.25;0.1]$ and the total attack cost $\Omega=2$. Next, we show the convergence error satisfies the submodular conditions, which are also validated through simulation results. When the selected attack set $\mathcal{A}=\mathcal{B}$, it is easy to obtain that $h(\mathcal{A} \cup \mathcal{B},\mathcal{B} \backslash \mathcal{A}) = 0$ and $\gamma=1$. Thus, (\ref{13}) and (\ref{add14}) are established. When the selected attack set $\mathcal{A} \subset \mathcal{B} \subseteq \mathcal{V}$, we have $h(\mathcal{A} \cup \mathcal{B},\mathcal{B} \backslash \mathcal{A}) = h(\mathcal{B},\mathcal{B} \backslash \mathcal{A})$. Assume the selected attack set $\mathcal{A}=\{1\}$ and $\mathcal{B}=\{1,3,4,5\}$. Then, (\ref{13}) is established since 
	\begin{equation*}
	\left\{
	          \begin{array}{lr}
	          h(\mathcal{B},\{3\})= 0.5008,\\ h(\mathcal{B},\{4\})= 0.3950,\\
	          h(\mathcal{B},\{5\})= 0.2846.
	          \end{array}
	\right.
	\end{equation*}
    For the adding agent $j \in \mathcal{V} \backslash \mathcal{B}=\{2,6\}$, we have
    \begin{equation*}
    \left\{
             \begin{array}{lr}
             h(\mathcal{A},\{2\})=0.1485,\\
             h(\mathcal{B},\{2\})=0.2622,\\
             h(\{2\},\{2\})=0.4052,\\
             \gamma=0.5841,
             \end{array}
    \right.
    \end{equation*}
    when $j=2$. It follows that (\ref{add14}) is established. Similarly, (\ref{add14}) is also established when $j=6$ since
    \begin{equation*}
    \left\{
             \begin{array}{lr}
             h(\mathcal{A},\{6\})=-0.0012,\\
             h(\mathcal{B},\{6\})=-0.0963,\\
             h(\{6\},\{6\})=0.3845,\\
             \gamma=0.5270.
             \end{array}
    \right.
    \end{equation*}
   In this way, we could verify (\ref{13}) and (\ref{add14}) for all $\mathcal{A} \subseteq \mathcal{B} \subseteq \mathcal{V}$ and $j\in \mathcal{V} \backslash \mathcal{B}$. Although the computing process is complex, it provides a feasible method to judge the submodularity of the convergence error.   
	
\end{example}

\subsection{The Design of FDI-ASSA}
In this part, we describe the core idea of the proposed algorithm FDI-ASSA and obtain the suboptimal subset solution to maximize the convergence error.

As shown in Algorithm $1$, we first create a candidate empty attack agent subset $\mathcal{\hat{A}}$ in step $1$. Then, the construction of $\mathcal{\hat{A}}$ is finished in step $2$-$11$ if and only if the total attack cost does not exceed the upper bound $\Omega$. In step $3$-$6$, each possible agent in $\mathcal{V'}$ is traversed and its marginal benefit is calculated. Since choosing the agent that can produce the maximum marginal returns in the limited attack cost budget is the main aim, we only need to concern with the maximum marginal benefit at average attack cost. Finally, the validation for the upper bound of attack cost is finished in step $12$-$14$, and the near-optimal subset $\hat{\mathcal{A}}$ is returned in step $15$.  

\begin{remark}
	The key idea of FDI-ASSA is to find the optimal agent that produces the maximum convergence error under the average attack cost through traversing the possible set $\mathcal{V'}$ for each iteration. 
	FDI-ASSA is scalable and can be applied to other attack strategies.
\end{remark}   
\begin{algorithm}[t]
	\caption{False Data Injection Attack Set Selection Algorithm (FDI-ASSA)}
	\KwIn {Attack cost $c(i)$ for all $i\in \mathcal{V}$; total attack cost $\Omega$; convergence time $t_c$; attack strategy $\mathbf{\theta}$; system matrix $M$.}
	\KwOut{Attack subset $\mathcal{\hat{A}}$}
	$\mathcal{\hat{A}}\leftarrow \emptyset$; $\mathcal{V'}\leftarrow \mathcal{V}$;
	$c(\mathcal{\hat{A}})\leftarrow 0$;\\
	\While {$\mathcal{V'}\notin \emptyset$ and $c(\mathcal{\hat{A}})\leq \Omega$}
	{
		\For {all $a\in \mathcal{V'}$}
		{
			$\mathcal{\hat{A}}_{\alpha} \leftarrow \mathcal{\hat{A}}\cup \{a\}$;\\
			$\mathrm{gain}_a \leftarrow f(\mathcal{\hat{A}}_{\alpha})-f(\mathcal{\hat{A}})$;
		}
		$a^* \leftarrow \mathrm{arg} \mathrm{max}_{a\in \mathcal{V'}}[\frac{\mathrm{gain}_a}{c(a)}]$;\\
		$\mathcal{\hat{A}} \leftarrow \mathcal{\hat{A}}\cup \{a^*\}$;\\
		$\mathcal{V'} \leftarrow \mathcal{V'}\backslash \{a^*\}$;\\
		Compute attack cost $c(\mathcal{\hat{A}})$;
	}
	
	\If {$c(\mathcal{\hat{A}})>\Omega$}
	{$\mathcal{\hat{A}} \leftarrow \mathcal{\hat{A}}\backslash \{a^*\}$}
	\Return{$\mathcal{\hat{A}}$}
\end{algorithm}

\subsection{Performance Analysis of FDI-ASSA}

Next, we analyze the performance of FDI-ASSA, which is shown in Theorem \ref{firstperformance}.
\begin{theorem}\label{firstperformance}
	For the optimization problem (\ref{11}), Algorithm $1$ returns the near-optimal subset $\hat{\mathcal{A}}$ such that 
	\begin{equation}\label{add21}
	\frac{f(\mathcal{\hat{A}})-f(\emptyset)}{f(\mathcal{A^*})-f(\emptyset)}\geq 1-e^{-\frac{c(\mathcal{\hat{A}})}{\Omega}},
	\end{equation}
	where $\mathcal{A^*}$ is the optimal subset.
	In addition, the worst-case running time satisfies $T_w \sim \mathcal{O}(|\Omega| |\mathcal{V}| n^3 m^3)$.
\end{theorem}
\begin{proof}
	Let $\mathcal{A^*}$ be an optimal solution in (\ref{11}) and $\Omega^* \triangleq c(\mathcal{A^*})$. Let $s_i$ be the $i$-th agent added in $\mathcal{S}$ during the $i$-th iteration of step $2$-$11$ and $\mathcal{S}_i \triangleq \{s_1,s_2,\cdots,s_i\}$.
	We consider Algorithm $1$'s step $2$-$11$ terminate after $l+1$ iterations.
	
	First, we prove that for $i=1,2,\cdots,l+1$, it holds
	\begin{align}\label{add22}
	f(\mathcal{A^*})-f(\mathcal{S}_{i-1})\leq \Omega^* \frac{f(\mathcal{S}_{i})-f(\mathcal{S}_{i-1})}{c(s_{i})}.
	\end{align}
	
	Due to the monotonicity of the convergence error in (\ref{add13}),
	\begin{align*}
	f(\mathcal{A^*})-f(\mathcal{S}_{i-1})\leq& f(\mathcal{A^*}\cup \mathcal{S}_{i-1})-f(\mathcal{S}_{i-1})\\
	=& f(\mathcal{A^*}\backslash \mathcal{S}_{i-1} \cup{\mathcal{S}_{i-1}})-f(\mathcal{S}_{i-1}).
	\end{align*}
	Let $\{z_1,z_2,\cdots,z_m\}\triangleq \mathcal{A^*}\backslash \mathcal{S}_{i-1}$, and
	{\small{
	\begin{align*}
	\beta_j\triangleq f(\mathcal{S}_{i-1} \cup \{z_1,z_2,\cdots,z_{j}\})-f(\mathcal{S}_{i-1} \cup \{z_1,z_2,\cdots,z_{j-1}\})
	\end{align*}}}for $j=1,2,\cdots,m$. Then, based on the [\textit{proposition $1.1$} \cite{fisher1978analysis}], we have
    \begin{align*}
    f(\mathcal{A^*})-f(\mathcal{S}_{i-1})\leq \sum_{j=1}^{m}\beta_j.
    \end{align*}
  Furthermore, we can infer  
	\begin{align*}
	\frac{\beta_j}{c(z_j)}\leq \frac{f(\mathcal{S}_{i-1}\cup \{z_{j}\})-f(\mathcal{S}_{i-1})}{c(z_j)} \leq \frac{f(\mathcal{S}_{i})-f(\mathcal{S}_{i-1})}{c(s_i)}. 
	\end{align*}
	Since $\sum_{j=1}^{m}c(z_j)\leq \Omega^*$,
	\begin{align*}
	f(\mathcal{A^*})-f(\mathcal{S}_{i-1})\leq \sum_{j=1}^{m}\beta_j\leq \Omega^* \frac{f(\mathcal{S}_{i})-f(\mathcal{S}_{i-1})}{c(s_{i})}.
	\end{align*}
	The above proof process is similar to Lemma $13$ in \cite{tzoumas2020lqg}. Based on (\ref{add22}), we have
	\begin{align*}
	f(\mathcal{S}_i)-f(\mathcal{S}_{i-1})
	\geq \frac{c(s_i)}{\Omega^*} [f(\mathcal{A^*})-f(\mathcal{S}_{i-1})].
	\end{align*}
	
	Next, we show the relationship between $f(\mathcal{S}_i)$ and $f(\mathcal{A^*})$.
	For $i=1,2,\cdots,l+1$,
	\begin{align}\label{add23}
	f(\mathcal{A^*})-f(\mathcal{S}_{i-1})-[f(\mathcal{A^*})-f(\mathcal{S}_{i})]
	\nonumber \\
	\geq \frac{c(s_i)}{\Omega^*} [f(\mathcal{A^*})-f(\mathcal{S}_{i-1})].
	\end{align}
	Let $\delta_{i-1}\triangleq f(\mathcal{A^*})-f(\mathcal{S}_{i-1})$. Then, (\ref{add23}) is transformed as
	\begin{align}\label{add24}
	\delta_{i-1}-\delta_i\geq \frac{c(s_i)}{\Omega^*} \delta_{i-1}.
	\end{align}
	From (\ref{add24}), it follows that $\delta_i\leq (1-\frac{c(s_i)}{\Omega^*})\delta_{i-1}$. Rearranging terms, we have 
	\begin{align*}
	\delta_i\leq \prod_{j=1}^{i} (1-\frac{c(s_i)}{\Omega^*}) f(\mathcal{A^*}).
	\end{align*}
	Then, it is inferred that
	\begin{align}\label{add25}
	f(\mathcal{S}_i)\geq [1-\prod_{j=1}^{i} (1-\frac{c(s_i)}{\Omega^*})]f(\mathcal{A^*}).
	\end{align}
	Since the convergence error is normalized, we have $f(\mathcal{\emptyset})=0$. Combining it with (\ref{add25}), we have
	\begin{align*}
	f(\mathcal{S}_l)-f(\mathcal{\emptyset}) \geq& [1-\prod_{j=1}^{l} (1-\frac{c(s_l)}{\Omega^*})][f(\mathcal{A^*})-f(\mathcal{\emptyset})]\\
	\geq & (1-e^{-\frac{c(s_l)}{\Omega^*}})[f(\mathcal{A^*})-f(\mathcal{\emptyset})]\\
	\geq & (1-e^{-\frac{c(s_l)}{\Omega}})[f(\mathcal{A^*})-f(\mathcal{\emptyset})],
	\end{align*}
	where the second inequality follows from Lemma $9$ in \cite{tzoumas2020lqg} with $\gamma=1$, and the third form that $\frac{c(s_i)}{\Omega^*}\geq \frac{c(s_i)}{\Omega}$, since $\Omega^*\leq \Omega$, which implies $1-e^{-\frac{c(s_l)}{\Omega^*}}\geq 1-e^{-\frac{c(s_l)}{\Omega}}$. The lower bound of the actual solution relative to the optimal value is received.
	
	Finally, we show the FDI-ASSA's running time. There are at most $|\Omega|$ iterations in step $2$-$11$ if all agents' attack cost is the same. And for each iteration, there are at most
	$|\mathcal{V}|$ iterations in step $3$-$6$ since the number of agents that can be selected from set $|\mathcal{V'}|$ decreases. In addition, the optimization process in step $3$-$6$ can be computed in $\mathcal{O}(n^3 m^3)$ per time step. Therefore, the total worst-case running time is $\mathcal{O}(|\Omega| |\mathcal{V}| n^3 m^3)$.
\end{proof}
\begin{remark}
	Theorem \ref{firstperformance} derives the gap between the obtained solution and the optimal one.
	In (\ref{add21}),  $f(\mathcal{\hat{A}})-f(\emptyset)$ quantifies the gain from selecting $\mathcal{\hat{A}}$, and the right-hand side of (\ref{add21}) guarantees that the marginal gain is close to the optimal $f(\mathcal{A^*})-f(\emptyset)$. Specifically, only when the attack cost $c(\mathcal{\hat{A}})=\Omega$, the algorithm returns the optimal subset. Meanwhile, we analyze the effects of the dimensions of the state, the number of agents, and the total attack cost on the running time to find the subset of the compromised agents, shown as $T_w$. 
\end{remark}
\begin{figure}[t]
	\centering
	\includegraphics[width=0.45\textwidth]{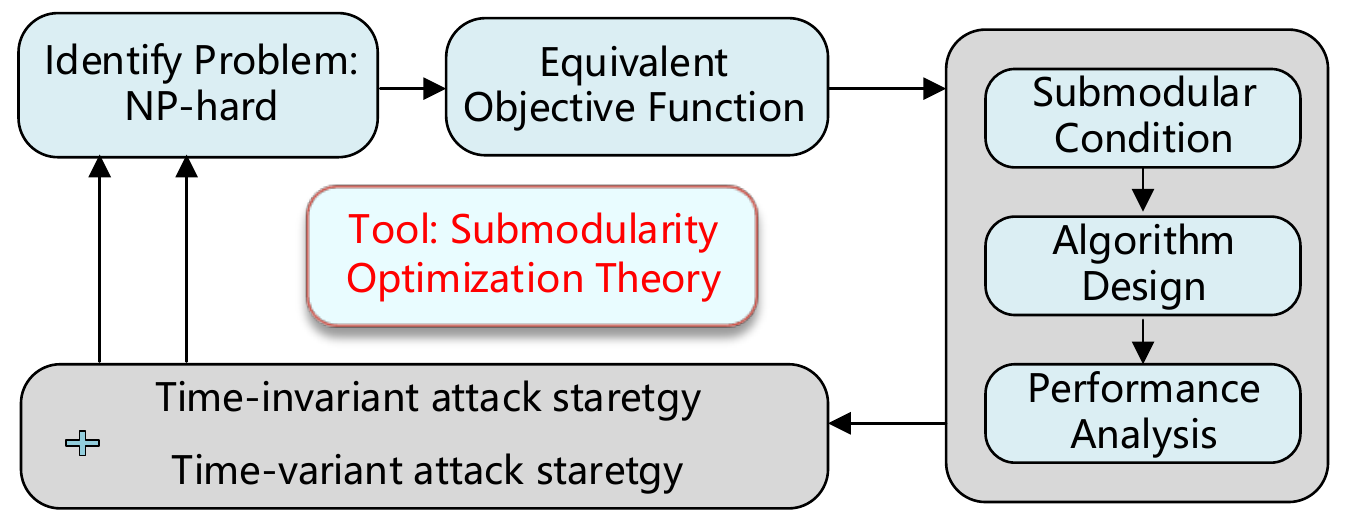}
	\caption{The analysis schematic of the FDI attack design problem}
	\label{structureflow}
	\vspace*{-10pt}
\end{figure}

\section{Time-variant attack set selection}\label{IV}
In this section, we extend the analysis of submodularity of the convergence error under the time-invariant attack strategy to the case under the time-variant attack strategy. In addition, we develop an improved false data injection attack set selection algorithm (IFDI-ASSA) to reduce the running time while approaching the optimal solution. In summary, the analysis diagram of the FDI attack design problem is shown in Fig. \ref{structureflow}.

\subsection{NP-hardness of Problem $\mathcal{P}_1$}
First, we show the complexity of the problem $\mathcal{P}_1$ under time-variant attacks.	
\begin{lemma}\label{NPhard1}
	\textbf{(Complexity)}.
	The problem $\mathcal{P}_1$ under the time-variant attack strategy $\mathbf{\theta}(t)$ is NP-hard.
\end{lemma}	
\begin{proof}
	First, we find the reduction of problem $\mathcal{P}_1$ under the time-variant attack strategy.
	Different from the proof in Lemma \ref{NPhard}, when $M=aI$ with $a<0$ being some constant and $\mathbf{\theta}(t)=[\theta_1(t), \cdots, \theta_m(t)]^{\mathrm{T}}$ with $\theta_i(t) \in \mathbb{R}_{\geq 0}$, the reduction of problem $\mathcal{P}_1$ by adding the compromised agent $i$ is denoted as
	\begin{align}
	f(\mathcal{A})_i=\{\sum_{j=m(i-1)+1}^{mi} (e_j^{\mathrm{T}} \int_{0}^{t} e^{-a(t-\tau)}  \mathbf{\theta}(\tau) d \tau )^2 \}^\frac{1}{2},
	\end{align}
	where $e_j \in \mathbb{R}^{m}$ is the canonical vector with $1$ in the $j$-th entry and $0$ elsewhere.
	
	Next, we find the corresponding instance of problem $\mathcal{P}_1$ under time-variant attacks for $0$-$1$ backpack problem. Given the set of values $\{\alpha_i\}$ and the other same parameters, when the stable system matrix $M=-I$ and the attack strategy $\mathbf{\theta}(t)=[\theta_1(t), \cdots, \theta_m(t)]^{\mathrm{T}}$ meets 
	\begin{align*}
	\sum_{j=m(i-1)+1}^{mi} (e_j^{\mathrm{T}} \int_{0}^{t} e^{-(t-\tau)} \mathbf{\theta}(\tau) d \tau )^2=\alpha_i^2
	\end{align*}
	for any $i\in \mathcal{A}$, we can see that the selected subset $\mathcal{A}$ for $0$-$1$ knapsack problem is optimal if and only if it is optimal for the corresponding problem $\mathcal{P}_1$. Thus, the proof is completed.
\end{proof}

\subsection{Submodularity under Time-variant Attacks}
\begin{theorem}\label{secondsubmodular}
	\textbf{(Submodularity)}.
	Under Assumptions $1$-$3$ with the time-variant attack strategy $\mathbf{\theta}(t)$, the convergence error $f(\mathcal{A})$ in problem $\mathcal{P}_1$ is normalized and monotone non-decreasing submodular. 
\end{theorem}

\begin{proof}
	The proof is divided into two parts. The first part is to prove that the convergence error is monotone non-decreasing. The second part is to show the submodularity of the convergence error.		
	
	\textbf{Monotonicity:} Under time-variant attacks, the convergence error can be written as
	\begin{align*}
	f(\mathcal{A})=\|\int_{0}^{t_c}e^{M(t-\tau)}\mathbf{\mu}^{\mathcal{A}}\otimes \mathbf{\theta}(\tau)d\tau\|=\|\kappa(\mathcal{A},t_c)\|.
	\end{align*}
	When $\mathcal{A}={\emptyset}$, we have $\mathbf{\mu}^{\mathcal{A}}=\mathbf{0}$. Thus, it is inferred that $f({\emptyset})=0$ and the convergence error is normalized. Then, for the set $\mathcal{A}\subseteq \mathcal{B} \subseteq \mathcal{V}$, we assume $\mathcal{A}=\{a_1,a_2,\cdots,a_s\}$ and $\mathcal{B}=\{a_1,a_2,\cdots,a_s,a_t\}$. It follows that	
	\begin{align*}
	f(\mathcal{A})=&\|\int_{0}^{t} e^{M(t-\tau)} \mathbf{\mu}^\mathcal{A} \otimes \mathbf{\theta}(\tau) d \tau \|\nonumber \\
	=&\{\sum_{l=1}^{mn} (e_l^{\mathrm{T}} \kappa(\mathcal{A},t_c))^2\}^{\frac{1}{2}},
	\end{align*}
	where $e_l \in \mathbb{R}^{mn}$ is denoted as the canonical vector with one in the $l$-th entry and $0$ elsewhere.
	
	For monotonicity, the requirement is that $f(\mathcal{B}) \geq f(\mathcal{A})$ always holds true when $\mathcal{A} \subseteq \mathcal{B}$. Since
	\begin{align*}
	f(\mathcal{B})=&\|\int_{0}^{t} e^{M(t-\tau)} (\mathbf{\mu}^{\mathcal{B} \backslash \mathcal{A}}+ \mathbf{\mu}^{\mathcal{A}})  \otimes \mathbf{\theta}(\tau) d \tau \|\\
	=& \{\sum_{l=1}^{mn} (e_l^{\mathrm{T}} \kappa(\mathcal{B},t_c))^2 \}^{\frac{1}{2}}\\
	\geq & \{\sum_{l=1}^{mn} (e_l^{\mathrm{T}} \kappa(\mathcal{A},t_c))^2 \}^{\frac{1}{2}},
	\end{align*}
	the monotonicity of the convergence error always holds.
	
	\textbf{Submodularity:} For submodularity, the requirement is that $f(\mathcal{A} \cup \{j\})-f(\mathcal{A})\geq f(\mathcal{B} \cup \{j\})-f(\mathcal{B})$ for the adding element $j\in \{a_{t+1},\cdots,a_{n}\}$. Let $\rho_j(\mathcal{A}) \triangleq f(\mathcal{A} \cup \{j\})-f(\mathcal{A})$ and $\rho_j^{+}(\mathcal{A}) \triangleq f(\mathcal{A} \cup \{j\})+f(\mathcal{A})$.
	This requirement can be rewritten as
	\begin{align*}
	\rho_j(\mathcal{A}) \geq \rho_j(\mathcal{B}).
	\end{align*}
	Since 
	\begin{align*}
	\rho_j(\mathcal{A}) \rho_j^{+}(\mathcal{A})=&\sum_{l=1}^{mn} (e_l^{\mathrm{T}} \kappa(\mathcal{A}\cup \{j\},t_c) )^2 -\sum_{l=1}^{mn} (e_l^{\mathrm{T}} \kappa(\mathcal{A},t_c) )^2\\
	=&2 \sum_{l=1}^{mn} (e_l^{\mathrm{T}} \kappa(\mathcal{A},t_c))(e_l^{\mathrm{T}} \kappa(\{j\},t_c))\\
	&+\sum_{l=1}^{mn} (e_l^{\mathrm{T}} \kappa(\{j\},t_c))^2\\
	=& \sum_{l=1}^{mn} (e_l^{\mathrm{T}} \kappa(\{j\},t_c))^2,
	\end{align*}
	where the second equality follows from the expansion of the sum of squares and the third equality follows from that 
	\begin{align*}
	(e_l^{\mathrm{T}} \kappa(\mathcal{A},t_c))(e_l^{\mathrm{T}} \kappa(\{j\},t_c))=0
	\end{align*}
	for any $l$. Similarly, we have 
	\begin{align*}
	\rho_j(\mathcal{B}) \rho_j^{+}(\mathcal{B})=\sum_{l=1}^{mn} (e_l^{\mathrm{T}} \kappa(\{j\},t_c))^2.
	\end{align*}
	Then, it is easy to infer
	\begin{align}
	\frac{\rho_j(\mathcal{A})}{\rho_j(\mathcal{B})}=\frac{\rho_j^{+}(\mathcal{B})}{\rho_j^{+}(\mathcal{A})}.
	\end{align}
	Based on the monotonicity of the convergence error, it is easy to infer $\rho_j^{+}(\mathcal{B})\geq \rho_j^{+}(\mathcal{A})$. Therefore, $\rho_j(\mathcal{A}) \geq \rho_j(\mathcal{B})$
	always holds and the submodularity of the convergence error is satisfied.
\end{proof}

Theorem \ref{secondsubmodular} shows the submodularity of the convergence error under time-variant attacks, which is a natural property and does not depend on the given time $t$, the system matrix $M$, and the attack strategy $\mathbf{\theta}(t)$. Instead of transforming the convergence error to an approximate value, the proof in Theorem \ref{secondsubmodular} provides a new insight to demonstrate the submodularity of the convergence error. Recalling the complex submodular conditions in Theorem \ref{firstsubmodular}, we can infer that the conditions always hold. 

\subsection{The Design of IFDI-ASSA}
In this part, we propose an improved FDI-ASSA (IFDI-ASSA) to obtain the suboptimal subset solution to maximize the convergence error, which is shown in Algorithm $2$. Compared with Algorithm $1$, the main differences of Algorithm $2$ lie in the adjustment of the circulation body (see Algorithm $1$' step $2$-$11$ and Algorithm $2$' step $2$-$17$), where Algorithm $2$ only calculates the marginal returns of the agents in the given total attack cost. 

	\begin{algorithm}[t]
	\caption{Improved False Data Injection Attack Set Selection Algorithm (IFDI-ASSA)}
	\KwIn {Attack cost $c(i)$ for all $i\in \mathcal{V}$; total attack cost $\Omega$; convergence time $t_c$; attack strategy $\mathbf{\theta}$; system matrix $M$.}
	\KwOut{Attack subset $\mathcal{\hat{A}}$}
	$\mathcal{\hat{A}}\leftarrow \emptyset$; $\mathcal{V'}\leftarrow \mathcal{V}$;
	$c(\mathcal{\hat{A}})\leftarrow 0$;\\
	\While {$\mathcal{V'}\notin \emptyset$ and $c(\mathcal{\hat{A}})< \Omega$}
	{
		$\mathcal{\hat{A}}_{\alpha} \leftarrow \mathcal{\hat{A}}$;\\
		\For {all $a\in \mathcal{V'}$}
		{
			\If {$c(\mathcal{\hat{A}} \cup \{a\})\leq \Omega$}
			{
				$\mathcal{\hat{A}}_{\alpha} \leftarrow \mathcal{\hat{A}}\cup \{a\}$;\\
				$\mathrm{gain}_a \leftarrow f(\mathcal{\hat{A}}_{\alpha})-f(\mathcal{\hat{A}})$;
			}
			
		}
		\If {$\mathcal{\hat{A}}_{\alpha} == \mathcal{\hat{A}}$}
		{
			break;
		}
		$a^* \leftarrow \mathrm{arg} \mathrm{max}_{a\in \mathcal{V'}}[\frac{\mathrm{gain}_a}{c(a)}]$;\\
		$\mathcal{\hat{A}} \leftarrow \mathcal{\hat{A}}\cup \{a^*\}$;\\
		$\mathcal{V'} \leftarrow \mathcal{V'}\backslash \{a^*\}$;\\
		Compute attack cost $c(\mathcal{\hat{A}})$;
	}
	
	\Return{$\mathcal{\hat{A}}$}
\end{algorithm}

\subsection{Performance Analysis of IFDI-ASSA}
In this part, we analyze the performance of IFDI-ASSA.
For the optimization problem (\ref{11}), IFDI-ASSA returns the near-optimal subset $\hat{\mathcal{A}}$ such that it provides a better suboptimal upper bound and a shorter running time than FDI-ASSA
even though it produces the same suboptimal lower bound and worst-case running time as Theorem \ref{firstperformance}. 

First, we explain why IFDI-ASSA provides a better suboptimal upper bound and a shorter running time. IFDI-ASSA aims to avoid traversing each agent in available set $\mathcal{V}'$ and only the agents whose participation does not exceed the total attack cost are access to computing the marginal returns (see step $5$-$8$ of Algorithm $2$). Compared with FDI-ASSA, IFDI-ASSA can guarantee all possible agents have been tried while reducing the running time when the attack cost of each agent is not the same. Specifically, when the attack cost of the selected optimal agent $a^{*}$ exceeds the total attack cost, it must be eliminated in Algorithm $1$' step $12$-$14$. At the same time, there exists another agent whose participation does not exceed the total attack cost, which should be added in the attack set $\hat{\mathcal{A}}$. IFDI-ASSA can compute the marginal returns of the agent, which further narrows down the field of computation and provide a larger convergence error. Unfortunately, FDI-ASSA cannot include it. An illustrative example is shown as Example $2$. 

\begin{example}
	There are six agents connected in the linear network, whose attack cost meet $c(1)=1$, $c(6)=1$, and $c(i)=2$ for $i=2,3,4,5$. Assume the total attack cost $\Omega=6$ and the selected attack set $\mathcal{\hat{A}}=\{1,2,3\}$ in the $3$-th iteration. For the next iteration (i.e., $4$-th iteration), if the optimal select agent $a^{*}=4$, the agent $i=4$ will be eliminated in step $11$-$13$ of Algorithm $1$. Thus, the obtained solution is not optimal and there is a deviation between the optimal and the actual convergence error. Nevertheless, in Algorithm $2$, only the agents whose participation does not exceed the given total attack cost can be accessed. As a result, the agent $i=6$ will be selected into $\mathcal{\hat{A}}$ to produce a better solution $\mathcal{\hat{A}}=\{1,2,3,6\}$ and obtain a larger convergence error. Moreover, Algorithm $2$ avoids computing the marginal returns of agents $i=4,5$, so that it speeds up the process of the algorithm. 
\end{example}

Next, we show why IFDI-ASSA guarantees the same suboptimal lower bound and worst-case running time as Theorem \ref{firstperformance}. Since IFDI-ASSA does not change the principle of selecting the added compromised agent for each iteration, the proof of suboptimal lower bound is the same as that in Theorem \ref{firstperformance}. Hence, the suboptimal lower bound of IFDI-ASSA maintains. Especially, the lower bound holds when the attack cost of each agent is the same. Then, we analyze the worst-case running time of IFDI-ASSA. There are at most $|\Omega|$ iterations in step $3$-$18$ if all agents' attack cost is the same. And for each iteration, there are at most
$|\mathcal{V}|$ iterations in step $5$-$10$ since the number of agents that can be selected from set $|\mathcal{V'}|$ decreases and at most
$|\Omega|$ iterations in step $6$-$9$. In addition, the optimization process in step $6$-$9$ needs to be calculated in $\mathcal{O}(n^3 m^3)$ per time step. Therefore, the total worst-case running time is still $\mathcal{O}(|\Omega| |\mathcal{V}| n^3 m^3)$.		

\section{Discussion}\label{V}
In this section, we discuss the impact of attack strategy $\mathbf{\theta}(t)$ with special statistical properties and heterogeneous agents on the analysis of the convergence error's submodularity.
\subsection{Attack Strategy with Special Statistical Properties}
In this part, we consider a type of attacks with special statistical properties, i.e., zero-mean attack vectors with associated covariance matrices $\Sigma$. Then the expected value of $\kappa(\mathcal{A},t)$ in (\ref{add10}) equals to zero. The optimization objective in (\ref{11}) can be reconstructed as $\max_{\mathcal{A}\subseteq \mathcal{V}} \lim_{t\to \infty} \mathbb{E}[\|\mathbf{x}(\mathbf{\mu}^{\mathcal{A}},t)- \mathbf{x}(t)\|] $. Since	
\begin{align}\label{28}
\lim_{t\to \infty} \mathbb{E} [\|\mathbf{x}(\mathbf{\mu}^{\mathcal{A}},t)- \mathbf{x}(t)\|] 
=& \lim_{t\to \infty} \sum_{i\in \mathcal{V}} \mathbb{E}[(\mathbf{x}_i(t)-\mathbf{b})^{2}]  \nonumber \\
=& \lim_{t\to \infty} \sum_{i\in \mathcal{V}} var(\mathbf{x}_i(t)), 
\end{align}
where $\mathbf{b}$ is the final state without attacks. Let $f(\mathcal{A})\triangleq \lim_{t\to \infty} \mathbb{E}[\|\mathbf{x}(\mathbf{\mu}^{\mathcal{A}},t)- \mathbf{x}(t)\|]$. To guarantee the submodularity of $f(\mathcal{A})$, we should meet $f(\mathcal{A}) \leq f(\mathcal{B})$ and $f(\mathcal{A} \cup \{j\})-f(\mathcal{A})\geq f(\mathcal{B} \cup \{j\})-f(\mathcal{B})$ for any $j\in \mathcal{V} \backslash \mathcal{B}$ with $\mathcal{A}\subseteq \mathcal{B}\subseteq \mathcal{V}$. Recall the fact that $\mathbf{x}_i(t)\in \mathbb{R}^{m}$ and each agent has its dynamics. The difficulties of theoretically analyzing the submodularity of $f(\mathcal{A})$ lie in the following two aspects. One is that $f(\mathcal{A})$ is the implicit function of the attack set $\mathcal{A}$ since the couple and dynamic relationships between the agent's state and the set of the compromised agents. The other is that it is difficult to find a method to transform $f(\mathcal{A})$ to an analyzable equivalent expression with respect to $\mathcal{A}$. There exist some works considering the transformation of the system error, defined as the mean-square error of the follower agent' states from their desired steady-state value, such as \cite{clark2013supermodular} and \cite{barooah2006graph}. Herein, it is shown that $\lim_{t\to \infty} var(\mathbf{x}_i(t))=\frac{1}{2} (L_{ff}^{-1})_{ii}$ and (\ref{28}) is equivalent to $\frac{1}{2} \sum_{i\in \mathcal{V}} (L_{ff}^{-1})_{ii}$ where $L_{ff}$ is the Laplacian matrix of the followers. However, note that \cite{clark2013supermodular} and \cite{barooah2006graph} analyzed the system error due to link noise and it is impossible for (\ref{28}) to be applied since the problem is different.

\subsection{Heterogeneous Agents}
Consider a system with $n$ mobile heterogeneous agents. The dynamics of the $i$-th agent for any $i \in \mathcal{V} $ are
\begin{equation}
\dot{\mathbf{x}}_i(t)=A_i \mathbf{x}_i(t)+B_i \mathbf{u}_i(t),  
\end{equation}
where $A_i$ and $B_i$ are constant matrices. For achieving the output consensus of heterogeneous agents without attacks, observer-based control rules \cite{lu2017cooperative} are usually designed such that the objective $\lim_{t\to \infty} \|\mathbf{y}_i(t)-\mathbf{y}_0(t)\|=0$ where $\mathbf{y}_i(t)$ is the system output and $\mathbf{y}_0(t)$ is the tracking output. Under attacks, the optimization objective in (\ref{11}) can be reconstructed as $\max_{\mathcal{A}\subseteq \mathcal{V}} \lim_{t\to \infty} \sum_{i\in \mathcal{V}} \|\mathbf{y}_i(t)-\mathbf{y}_0(t)\|$. Let $f(\mathcal{A})\triangleq \lim_{t\to \infty} \sum_{i\in \mathcal{V}} \|\mathbf{y}_i(t)-\mathbf{y}_0(t)\|$. If we can find the equivalent expression of $f(\mathcal{A})$ and prove its submodularity, then the proposed algorithms in this paper can be applied to find the suboptimal attack set to maximize the error directly.

\section{Simulation Results}\label{VI}
In this section, we evaluate the performance of the proposed algorithms, i.e., we analyze the impact of the subset of the compromised agents with two-dimensional states and three-dimensional states on convergence error, respectively. Especially, the comparison of Algorithm $2$ (IFDI-ASSA) and Algorithm $1$ (FDI-ASSA) is in Section \ref{Algorithms}.
\subsection{Agents with Two-dimensional States}\label{VI-A}
\subsubsection{Simulation Setup} We consider a consensus process with six agents. For each agent, its dynamics satisfy (\ref{3}) with two-dimensional state variables including relative position and velocity \cite{yu2010some}. Concretely, the local system matrix and input matrix are respectively set as $A=[-0.5~0;1~-1]$ and $B=[0.1~0.1;0.5~0.2]$. It is easily validated that the considered system is controllable and all eigenvalues of the global system matrix in (\ref{4}) are in the left half-plane. Without attacks, the system will achieve consensus asymptotically and we set $t_c=30s$. Since the weight entry $\bar{d}$ will influence the convergence rate, we consider choosing a moderate value to achieve consensus and let $\bar{d}=0.25$. If not specified otherwise, we consider two types of attack cost. One is the same attack cost for each agent (i.e., $c(i)=1$ for any $i\in \mathcal{V}$). The other is that the attack cost for each agent depends on its degree (i.e., $c(i)=d_i$ for any $i\in \mathcal{V}$) since the greater the degree, the more difficult it is to cut its communication link with neighbors. In addition, we set initial states $\mathbf{x}(0)=[-7~-3~-2~-2~0~1~1~-1~2~3~6~2]$, which do not affect the convergence error in this paper. The time-invariant attack satisfies $\mathbf{\theta}(t)=K=[0.25;0.1]$ and the time-variant attack is denoted as $\mathbf{\theta}(t)=K\mathrm{sin}(t)$. Under these parameters, the submodularity of the convergence error is satisfied.

\subsubsection{Compared Algorithms}
We compare four algorithms including the brute force search, random, degree-based, and our algorithm. All algorithms only differ by the selections of the compromised agents. In the brute force search, we enumerate all subsets in the constraint of the required attack cost, and the optimal subset of compromised agents is chosen to act as adversaries. In the random algorithm, a random subset of agents is chosen to serve as adversaries. In the degree-based algorithm, agents are selected to act as adversaries according to the degree, from the highest degree (i.e., the largest number of neighbors) to the lowest degree. Our algorithm is Algorithm $1$, which is also named as the submodular optimization algorithm in this section. 
\subsubsection{Results}
First, under the linear network, we analyze the performance of the four algorithms considered for the problem of choosing the subset of compromised agents. Especially, the results are reported in Fig.\ref{fig1} for the case where all agents are manipulated with the same attack cost, and in Fig. \ref{fig2} for the case where the attack cost of each agent depends on its degree. When the adversary launches the time-invariant attack, Fig. \ref{fig1.a} shows the effectiveness of the submodular optimization algorithm, which results in the maximum convergence error than other algorithms despite varying total attack costs. In Fig. \ref{fig1.b}, either the brute force search or the submodular optimization algorithm provides the greater convergence error when there exists the time-variant attack, which is still independent of limited total attack costs. Therefore, the submodular optimization algorithm almost matches the brute force search while reducing computational complexity than it, which outperforms in terms of the convergence error among the provided random and degree-based algorithms.   

Fig. \ref{fig2.a} shows the stepwise convergence error with the time-invariant attacks under both random and degree-based algorithms while the submodular optimization algorithm still provides the same performance as the brute force search. On the other hand, the results under the time-variant attack, shown in Fig. \ref{fig2.b}, also suggest that Algorithm $1$ outperforms both random and degree-based algorithms despite the varying total attack cost and different attack strategies affect the size of the convergence error.

Next, we further consider the effects of network structures on the convergence error since the convergence error is also a function of the network structures. The linear graph, the cycle graph, and the fixed graph where $L=[1~ -1~ 0~ 0~ 0~ 0;-1~ 3~ -1~ -1~ 0~ 0;0~ -1~ 2~ 0~ -1~ 0~ 0;0~ -1~ 0~ 2~ -1~ 0;0~ 0~ -1~ -1~ 3~ -1;0~ 0~ 0~ 0~ -1~ 1]$ connected with six agents are analyzed for the problem of selecting the subset of compromised agents with different attack costs in the submodular optimization algorithm. The results are plotted in Fig. \ref{fig4}. From Fig. \ref{fig4.a}, the convergence errors in both linear graph and cycle graph are greater than that in the fixed graph. Similarly, adding a compromised agent in the linear graph leads to the worst performance compared to the cycle graph and fixed graph when the attack cost satisfies $c(i)=d_i$, shown in Fig. \ref{fig4.b}. Note that the convergence error decreases as the total attack cost increases under the fixed graph, implying that the submodularity of the convergence error is not satisfied. Furthermore, it validates the impact of network structure on submodularity. 

\begin{figure}[t]
	\centering
	\subfigure[]{\label{fig1.a}
		\includegraphics[width=0.225\textwidth]{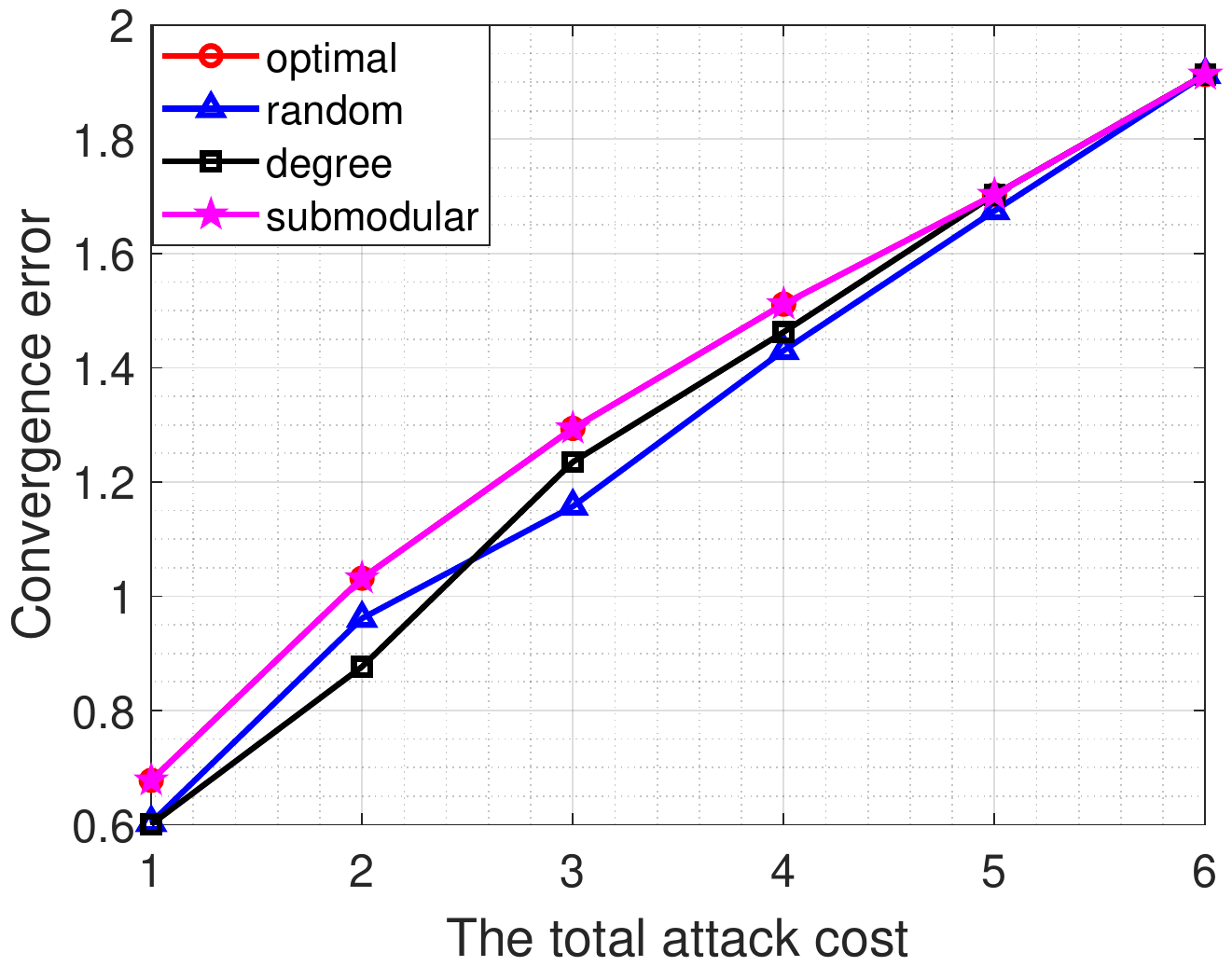}
	}
	\subfigure[]{\label{fig1.b}
		\includegraphics[width=0.225\textwidth]{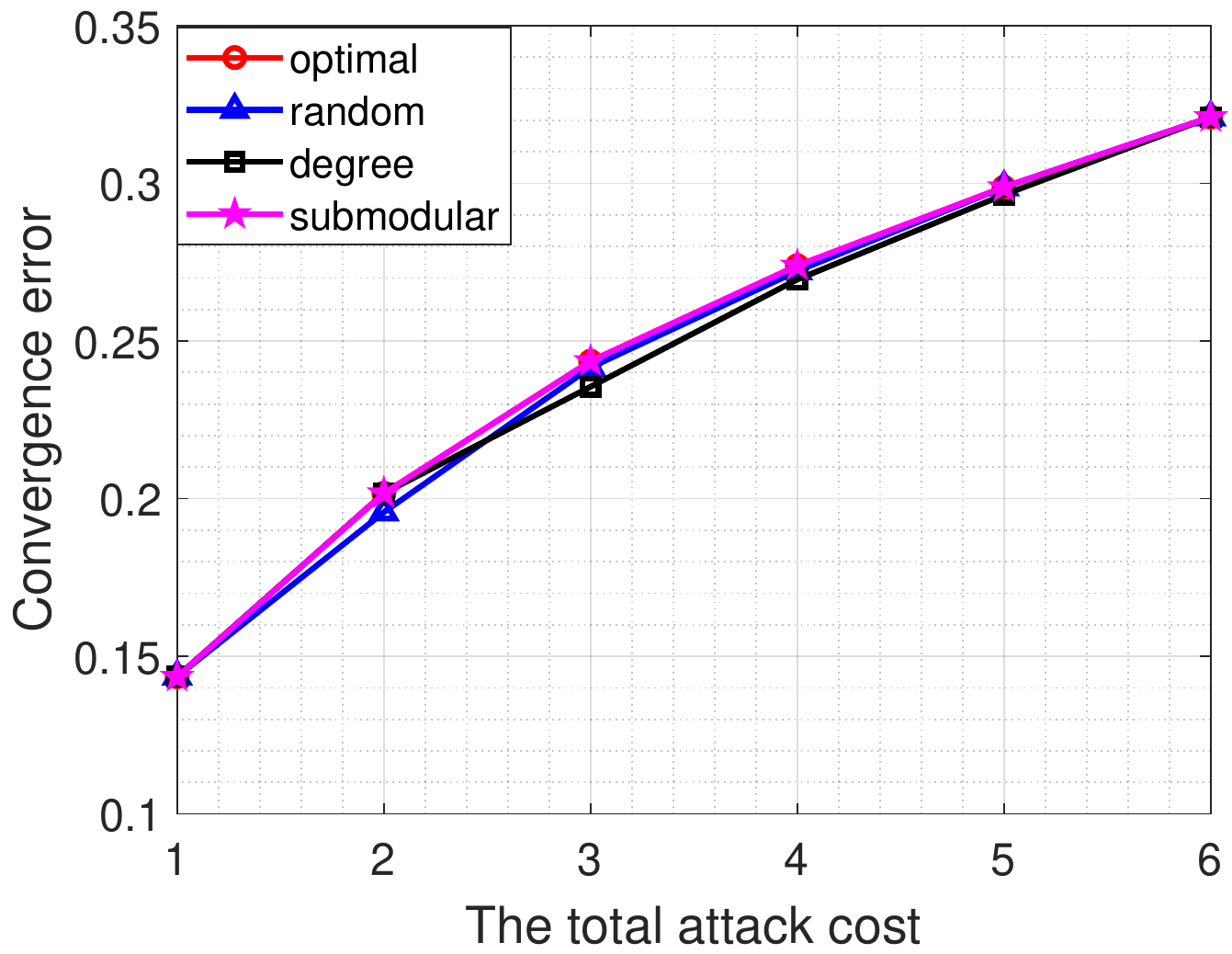}	
	}
	\caption{The convergence error with attack cost $c(i)=1$ for any $i\in \mathcal{V}$ in the linear graph. (a) The time-invariant attack (b) The time-variant attack}
	
	\label{fig1}
	\vspace*{-10pt}
\end{figure}

\begin{figure}[t]
	\centering
	\subfigure[]{\label{fig2.a}
		\includegraphics[width=0.225\textwidth]{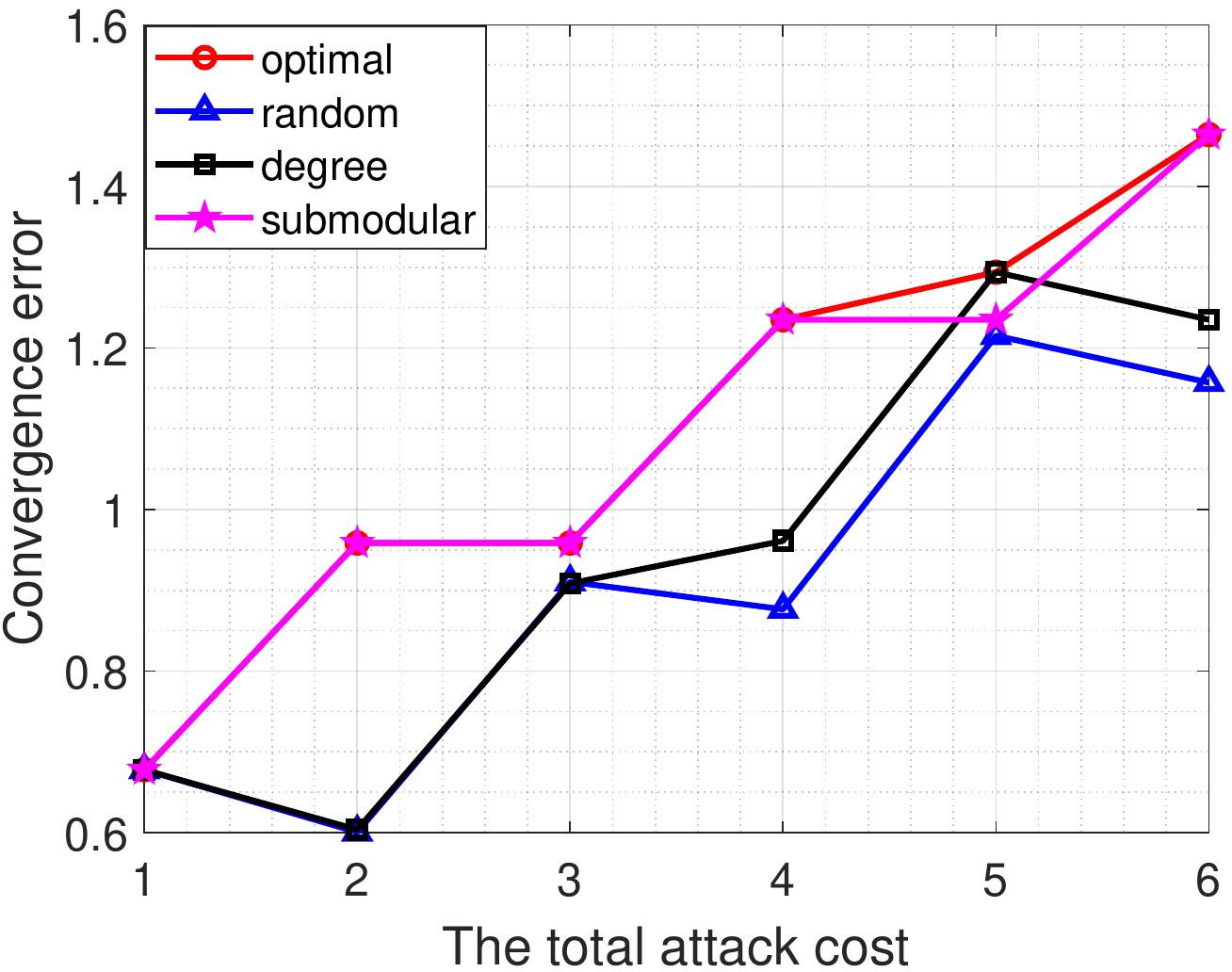}
	}
	\subfigure[]{\label{fig2.b}
		\includegraphics[width=0.225\textwidth]{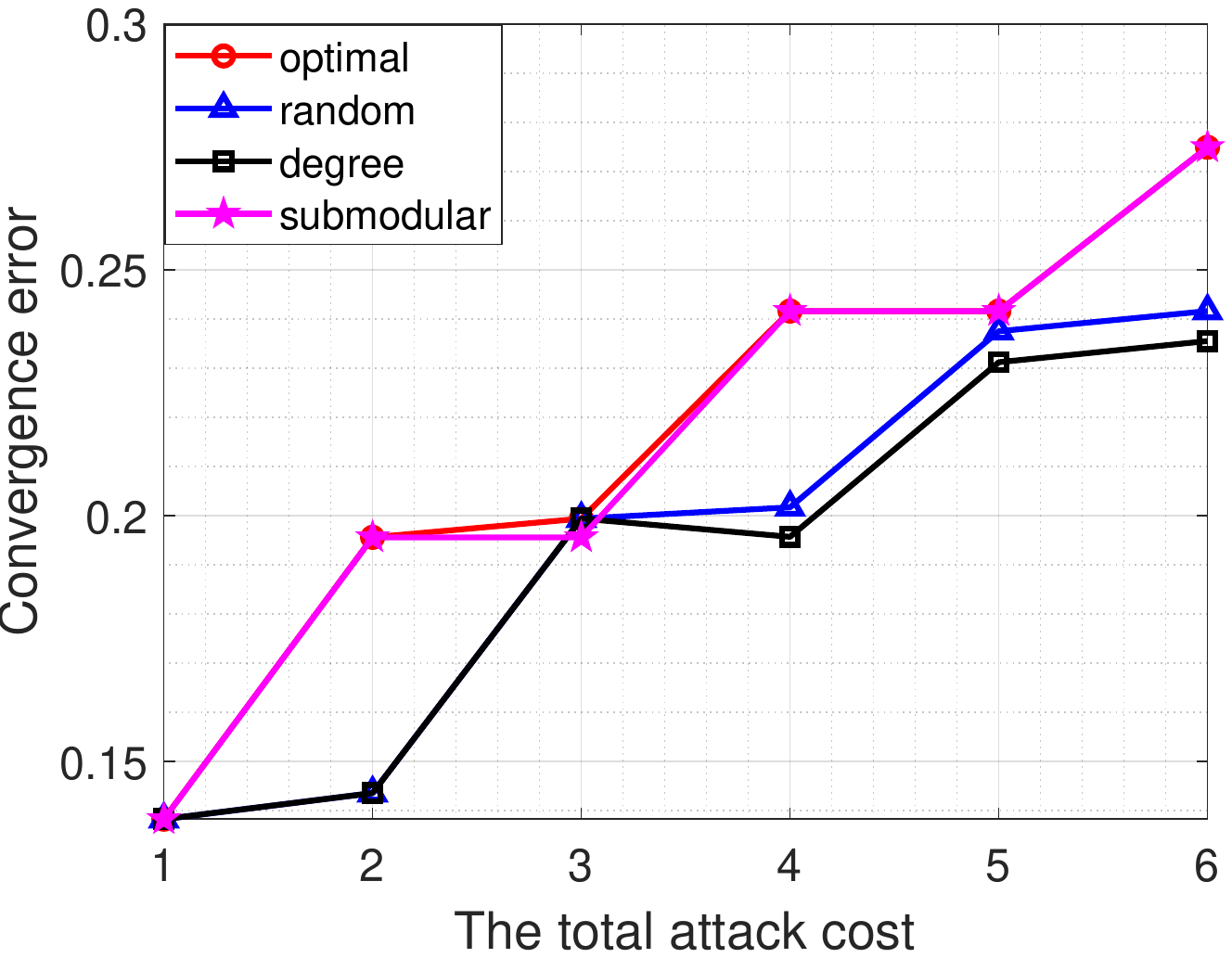}	
	}
	\caption{The convergence error with attack cost $c(i)=d_i$ for any $i\in \mathcal{V}$ in the linear graph. (a) The time-invariant attack (b) The time-variant attack}
	
	\label{fig2}
	\vspace*{-10pt}
\end{figure}

	\begin{figure}[t]
	\centering
	\subfigure[]{\label{fig4.a}
		\includegraphics[width=0.225\textwidth]{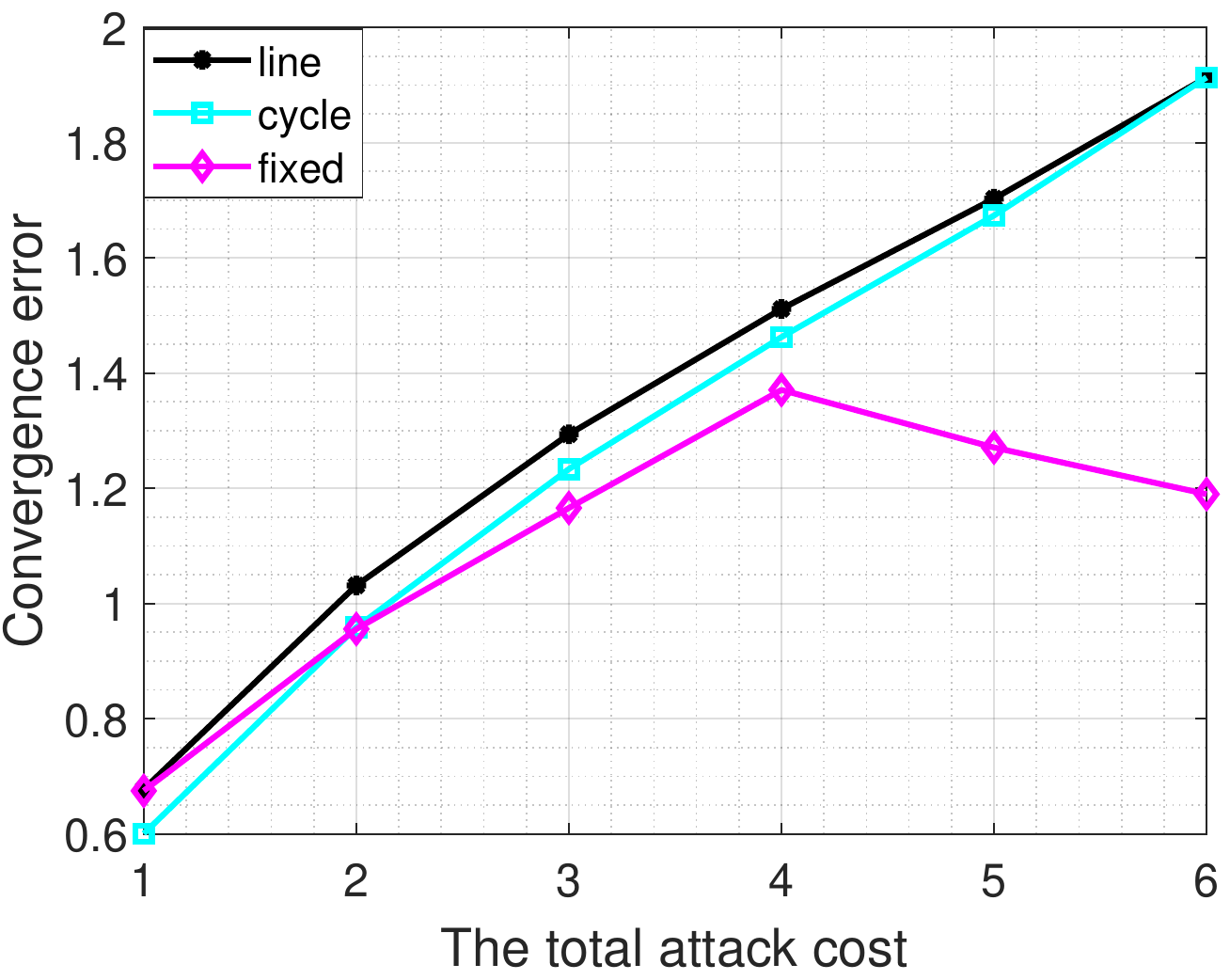}
	}
	\subfigure[]{\label{fig4.b}
		\includegraphics[width=0.225\textwidth]{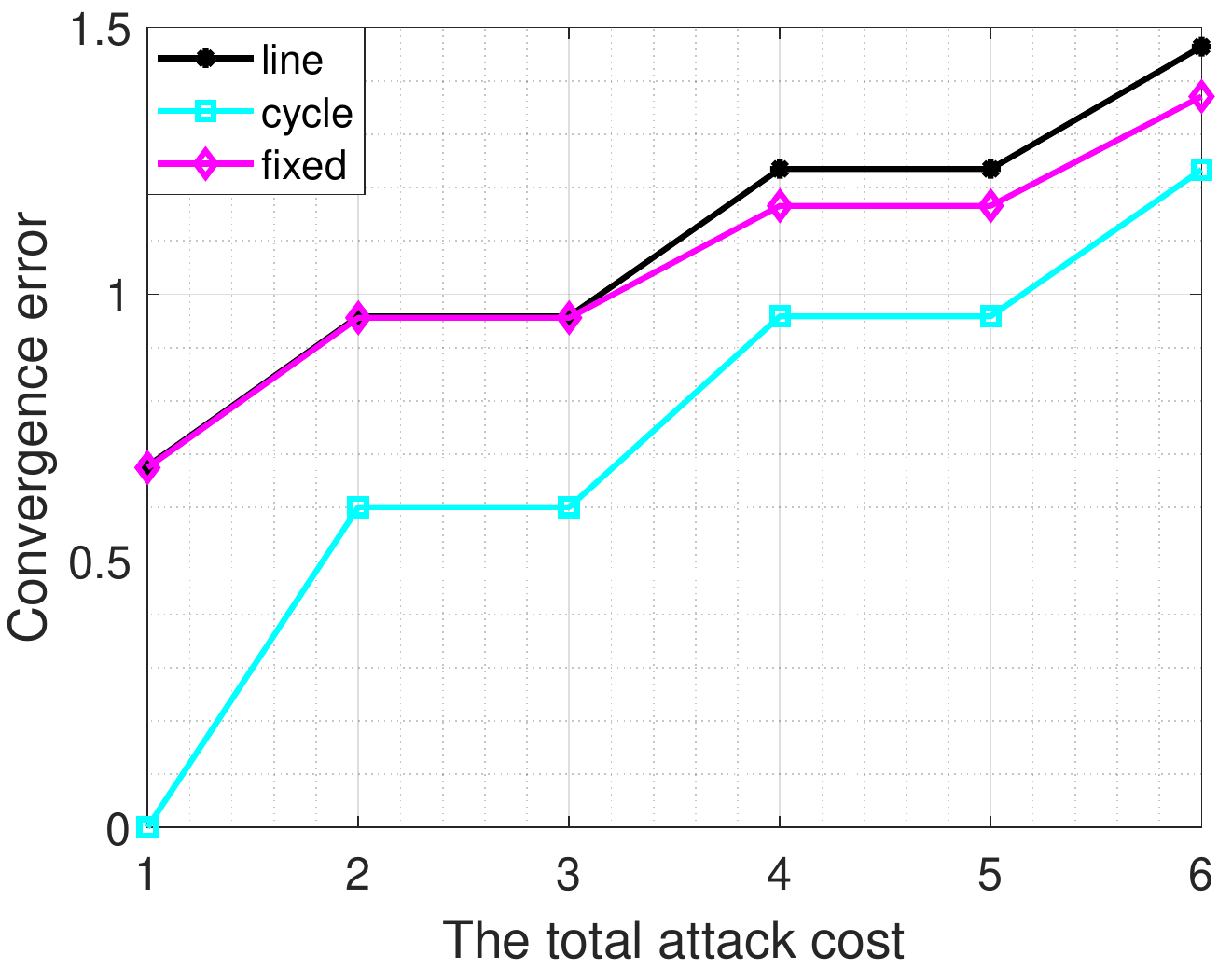}	
	}
	\caption{The convergence error with different network structures under the submodular optimization algorithm. (a) Attack cost $c(i)=1$ (b) Attack cost $c(i)=d_i$}
	
	\label{fig4}
	\vspace*{-10pt}
\end{figure}

\subsection{Agents with Three-dimensional States}
\subsubsection{Simulation Setup} In this part, we consider a network with $50$ autonomous vehicles (i.e., agents) randomly deployed in a $100\times 100m$ area, and the communication distance of each agent is $15m$. The dynamics of each agent satisfy (\ref{3}) with three-dimensional states, including the rate of the pitch angle, the pitch angle, and the depth \cite{rezaee2020almost}. Concretely, the local system matrix and input matrix are respectively set as $A=[-0.4037~-0.2052~0;-0.684~-0.8825~0;-0.1175~-0.2875~-0.3]$ and $B=[0.02394~0~0;0~0.01371~0;0~0~-0.00146]$.
The time-invariant attack satisfies $K=[0.25;0.1;0.2]$ and other parameters are set as before.
\subsubsection{Compared Algorithms}
We consider the four algorithms proposed in the previous part.
\subsubsection{Results}
The results are plotted in Fig. \ref{fig5.a} for the case where each agent has the same attack cost, and the network structure is shown in Fig. \ref{fig6.a} where the green nodes are the selected compromised agents when the total attack cost satisfies $\Omega=6$. As the total attack cost rises from $\Omega=1$ to $\Omega=6$, the trajectory of the selected subset is as follows
\begin{align*}
\{7\}\rightarrow \{7~ 12\} \rightarrow \{7~ 12~ 19\} \rightarrow\{7~ 12~ 19~ 50\}\\
\rightarrow \{7~ 12~ 19~ 29~ 50\} \rightarrow \{7~ 12~ 19~ 27~ 29~ 50\}.
\end{align*}
In Fig. \ref{fig5.b}, the results are for the case where the attack cost of each agent is its degree (i.e., different attack cost) and the network structure is shown in Fig. \ref{fig6.b} where the total attack cost satisfies $\Omega=6$. As the total attack cost rises from $\Omega=1$ to $\Omega=6$, the trajectory of the selected subset is as follows
\begin{align*}
\{24\}\rightarrow \{24~ 44\} \rightarrow \{24~ 44\} \rightarrow\{2~24~ 44\}\\
\rightarrow \{2~24~ 44\} \rightarrow \{2~5~24~ 44\}.
\end{align*}
In large scale multi-agent dynamical systems, the proposed algorithm still produces a larger convergence error than the random and the degree-based algorithms, which almost matches the optimal solution under the brute force search.
\begin{figure}[t]
	\centering
	\subfigure[]{\label{fig5.a}
		\includegraphics[width=0.225\textwidth]{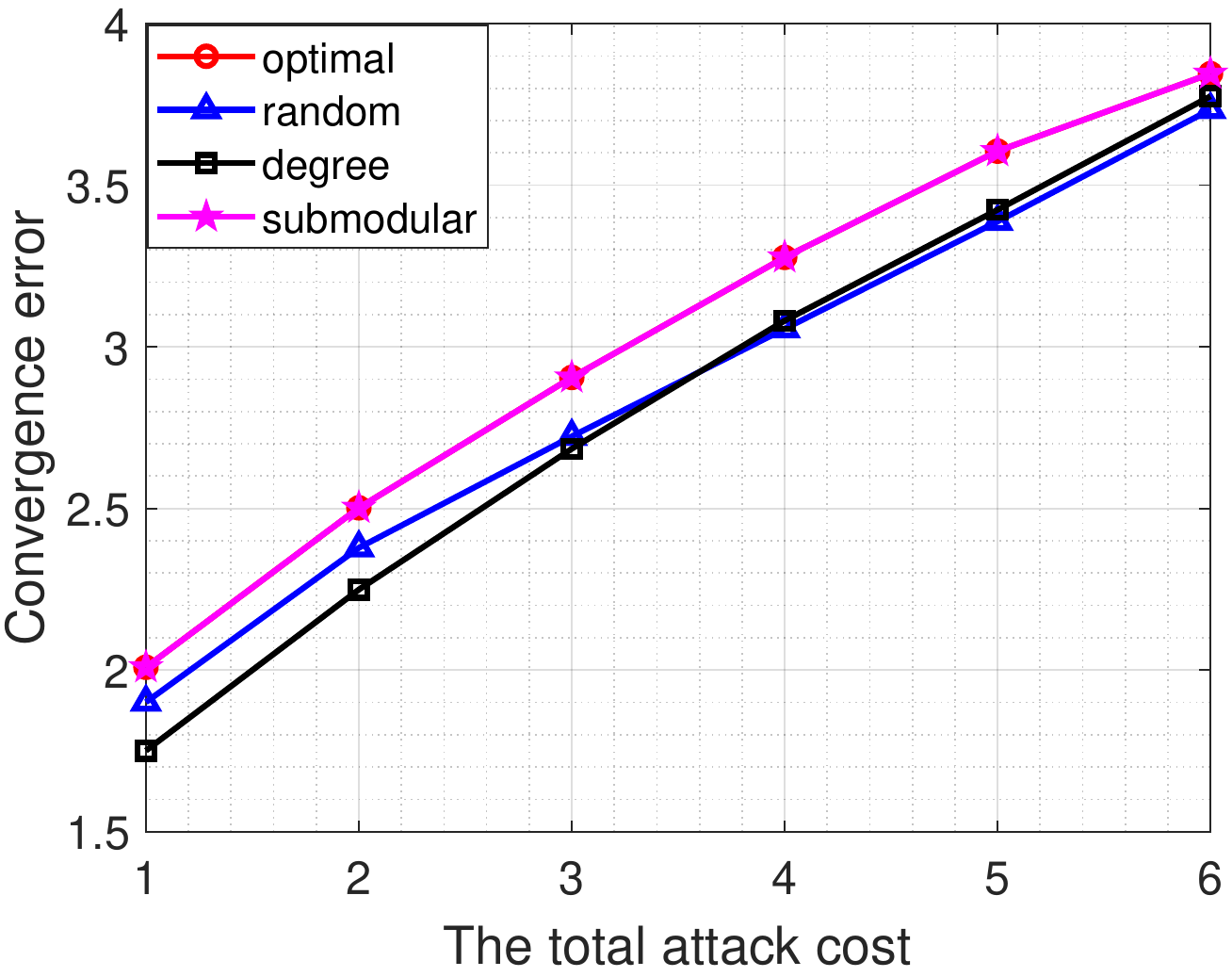}
	}
	\subfigure[]{\label{fig5.b}
		\includegraphics[width=0.225\textwidth]{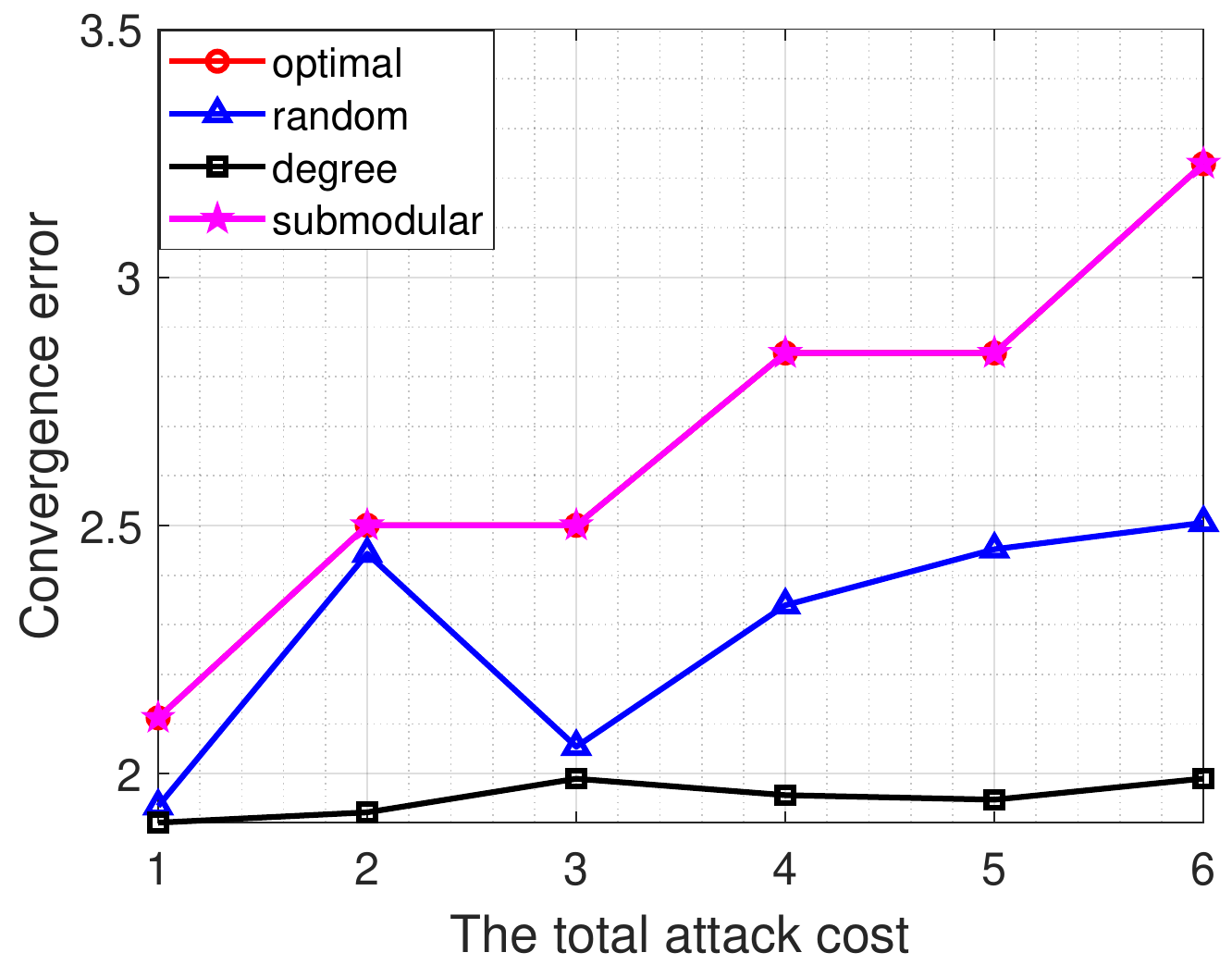}	
	}
	\caption{The convergence error with the time-invariant attack under the random network. (a) Attack cost $c(i)=1$ (b) Attack cost $c(i)=d_i$}
	
	\label{fig5}
	\vspace*{-10pt}
\end{figure}

\begin{figure}[t]
	\centering
	\subfigure[]{\label{fig6.a}
		\includegraphics[width=0.225\textwidth]{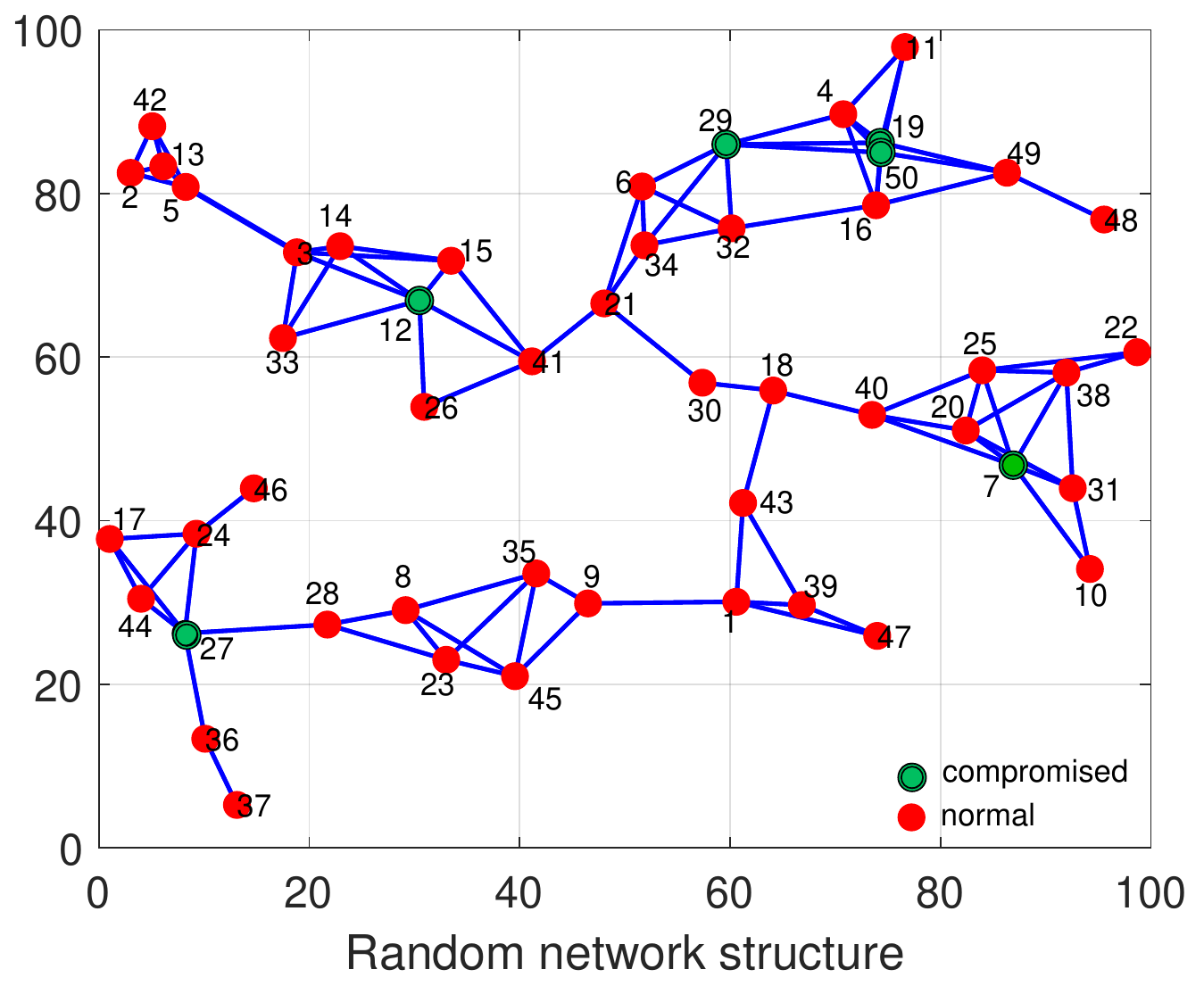}
	}
	\subfigure[]{\label{fig6.b}
		\includegraphics[width=0.225\textwidth]{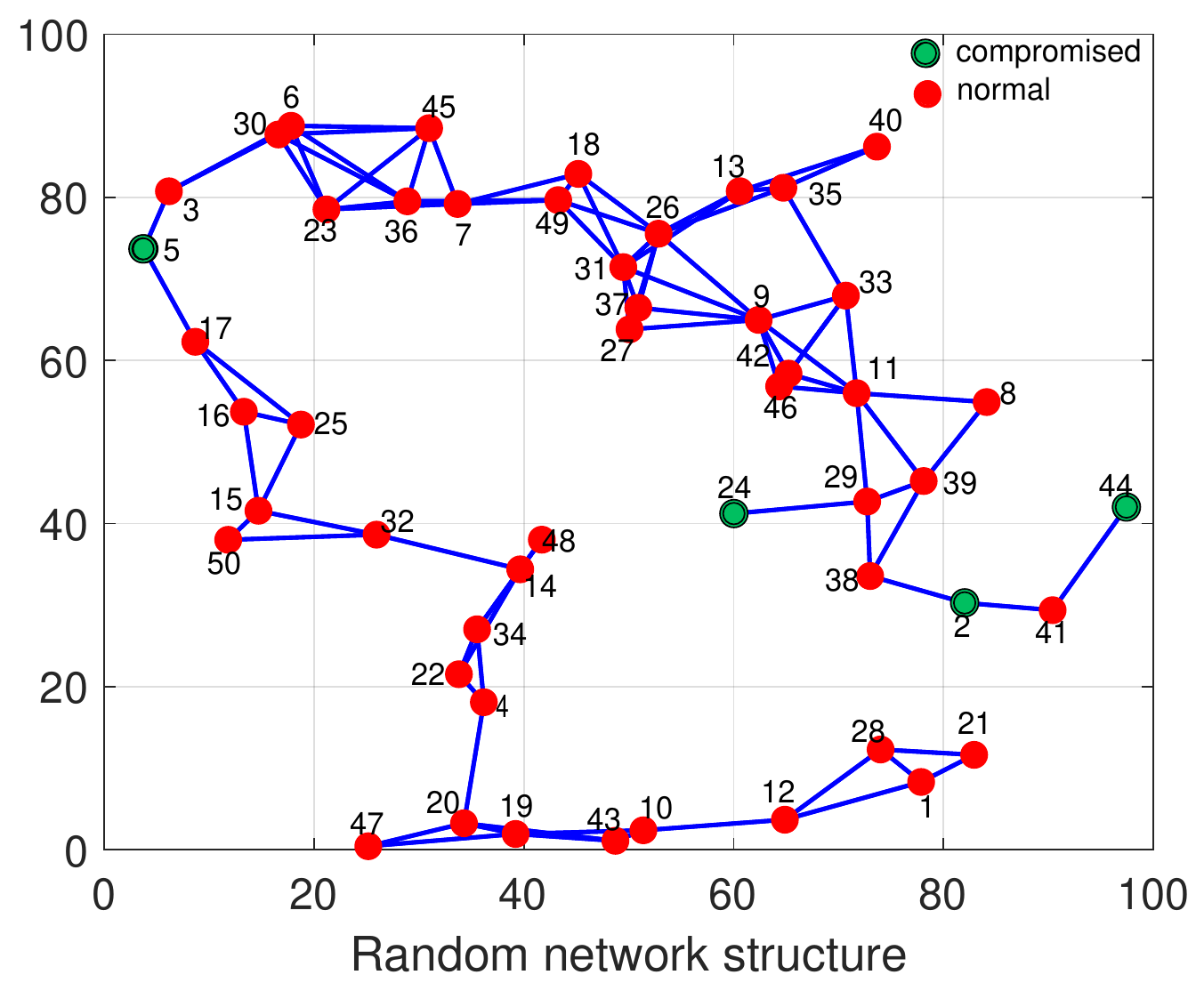}	
	}
	\caption{Different random network structures. (a) Attack cost $c(i)=1$ (b) Attack cost $c(i)=d_i$}
	
	\label{fig6}
	\vspace*{-10pt}
\end{figure}

\begin{figure}[t]
	\centering
	\subfigure[]{\label{fig7a}
		\includegraphics[width=0.35\textwidth]{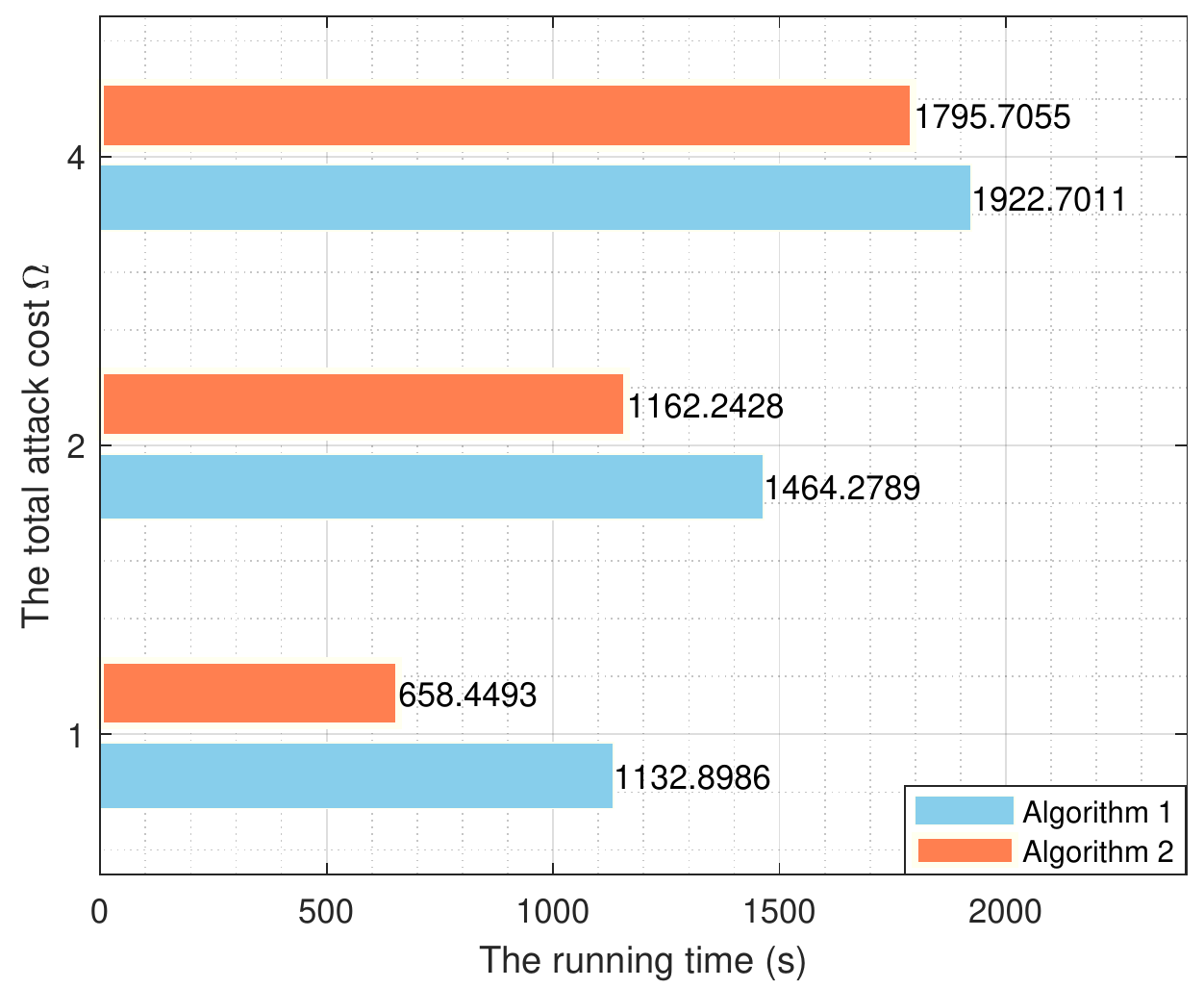}
	}
	
	\subfigure[]{\label{fig7b}
		\includegraphics[width=0.35\textwidth]{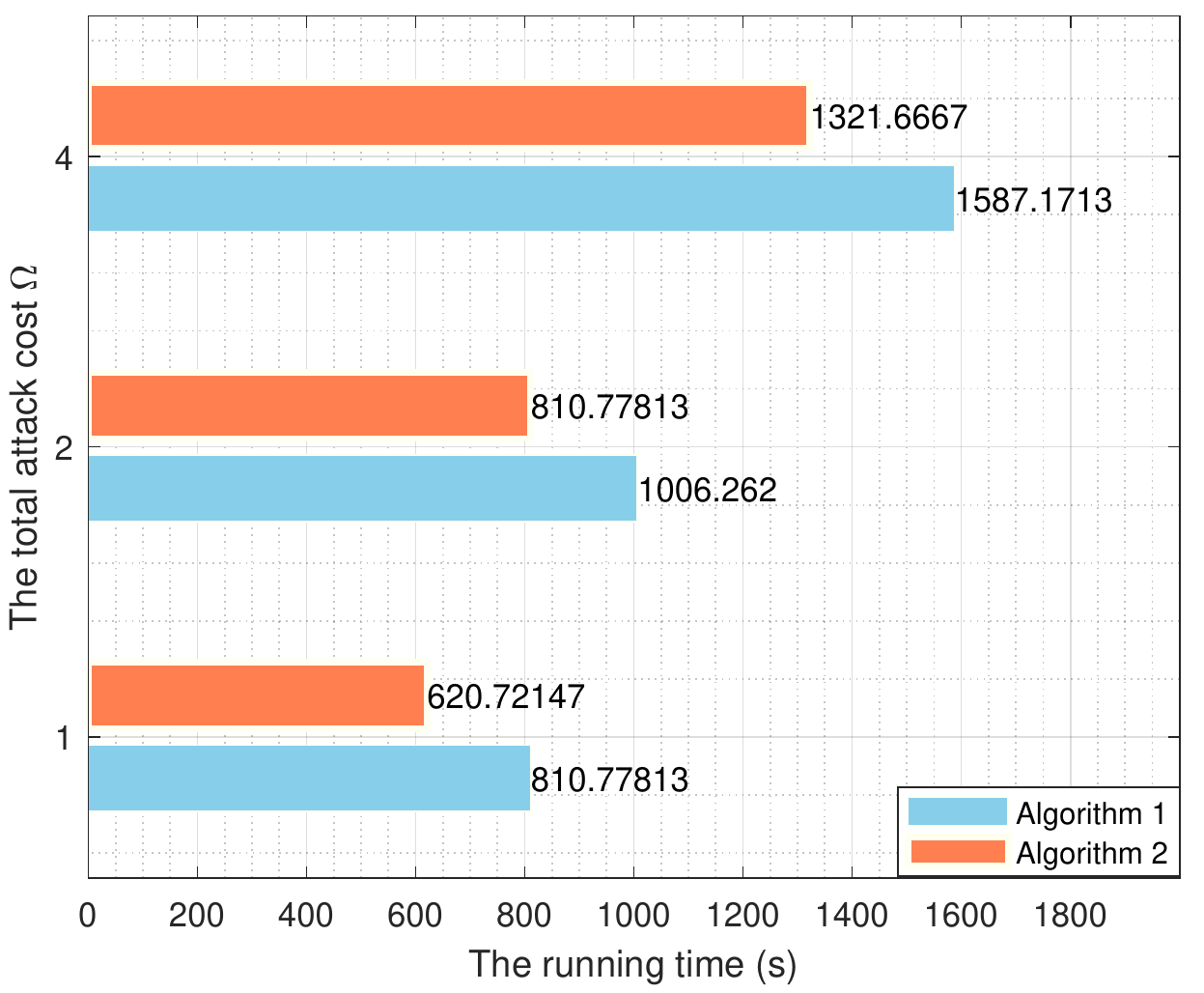}
	}
	\caption{Comparison results on the running time of Algorithm $1$ and $2$ under time-invariant attacks. (a) Attack cost $c(i)=1$ (b) Attack cost $c(i)=d_i$}
	
	\label{fig7}
	\vspace*{-10pt}
\end{figure}

\subsection{Comparison Results under Different Attack Strategies}
All basic parameter settings are set as the same as Section \ref{VI-A}.    
First, we validate the results of the attack strategy $K \mathrm{cost}(t)$. We choose $t_c=30 \in [9\pi, 10\pi]$ and $t_c=60 \in [19\pi, 20\pi]$ where $t_c=30$ is in the second half and $t_c=60$ is in the first half of cycles of $2\pi$. When the attack cost $c(i)=1$ for any $i\in \mathcal{V}$, we emphasize on the results of the total attack cost $\Omega=2$ for the sake of analysis. Based on Algorithm $1$, the obtained solution $\mathcal{\hat{A}}=\{1,2\}$. Let $\mathcal{A}=\{1,2\}$, $\mathcal{B}=\{1,2,3\}$ and the added agent $j=4$. When $t_c=30$, the convergence error $f(\{1,2\})=0.1905$, $f(\{1,2,3\})=0.2312$, 
$f(\{1,2,4\})=0.2375$, and
$f(\{1,2,3,4\})=0.2658$. Thus, the marginal returns $\rho_4(\{1,2\})=0.2375-0.1905=0.047$ and $\rho_4(\{1,2,3\})=0.2658-0.2312=0.0346$. It is easy to infer that there exists $j=4$ such that $\rho_j(\mathcal{A})> \rho_j(\mathcal{B})$ and the submodularity of the convergence error is established when $t_c=30$. Similarly, when $t_c=60$, we have the attack set solution $\mathcal{\hat{A}}=\{3,5\}$ and $\rho_j(\mathcal{A})\geq \rho_j(\mathcal{B})$ always holds true for $j\in \mathcal{V} \backslash \mathcal{B}$ in all possible set $\mathcal{A} \subseteq \mathcal{B} \subseteq \mathcal{V}$. 

Next, we analyze the results of attack strategy $\mathbf{\theta}(t)=K \mathrm{sin}(t)$ and $\mathbf{\theta}(t)=Ke^{-t}$. We find that the submodularity of the convergence error will not be effected by the given convergence time when $\mathbf{\theta}(t)=K \mathrm{sin}(t)$ and $\mathbf{\theta}(t)=Ke^{-t}$, while the selected attack set is different when $\mathbf{\theta}(t)=K\mathrm{sin}(t)$. In other words, if the system suffers from attack strategy $\mathbf{\theta}(t)=K\mathrm{sin}(t)$ and $\mathbf{\theta}(t)=Ke^{-t}$, the convergence error is submodular without constraints, which becomes the nature property of the system.

Finally, we show the results when the attack strategy has special statistical properties. the function \emph{randn} returns a sample of normally distributed random numbers with mean $0$ and variance $1$.
All results are summarized in Table \ref{table1}, where the data under a stochastic attack strategy represents an ensemble average of $1000$ trials. We find that the convergence error does not satisfy the submodularity when the attack strategy is stochastic, while the constant attack $K$ does not influence its submodularity.

\begin{table}[h]
	\setlength{\abovecaptionskip}{-12pt}
	\caption{THE IMPACT OF DIFFERENT ATTACK STRATEGIES ON SUBMODULARITY OF THE CONVERGENCE ERROR}
	\label{table1}
	\begin{center}
		\begin{tabular}{c c c c}
			\specialrule{0.15em}{3pt}{3pt}
			\tabincell{c}{\textbf{Attack}\\ \textbf{form}} 
			& \tabincell{c}{\textbf{Convergence} \\  \textbf{time} $t_c/s$}
			& \tabincell{c}{\textbf{Total attack} \\ \textbf{cost} $\Omega=2$}
			& \textbf{Submodular}\\

			\specialrule{0.05em}{3pt}{3pt}
			\multirow{4}{*}{$K$}
			& $30$
			&\tabincell{c}{$\mathcal{\hat{A}}=\{1,2\}:$\\ $f(\mathcal{\hat{A}})=1.0315$}
			& $\checkmark$\\
			\specialrule{0.0em}{3pt}{3pt}
			& $60$ 
			&  \tabincell{c}{$\mathcal{\hat{A}}=\{1,2\}:$\\ $f(\mathcal{\hat{A}})=1.0315$}
			&$\checkmark$ \\

			\specialrule{0.05em}{3pt}{3pt}
			\multirow{4}{*}{$K \mathrm{cost}(t)$}
			& $30$
			&\tabincell{c}{$\mathcal{\hat{A}}=\{1,2\}:$\\ $f(\mathcal{\hat{A}})=0.1905$}
			&$\checkmark$\\
			\specialrule{0.0em}{3pt}{3pt}
			& $60$ 
			&  \tabincell{c}{$\mathcal{\hat{A}}=\{3,5\}:$\\ $f(\mathcal{\hat{A}})=0.2948$}
			&$\checkmark$ \\

			\specialrule{0.05em}{3pt}{3pt}
			\multirow{4}{*}{$K \mathrm{sin}(t)$}
			& $30$
			& \tabincell{c}{$\mathcal{\hat{A}}=\{3,5\}:$\\ $f(\mathcal{\hat{A}})=0.2017$} 
			& $\checkmark$\\
			\specialrule{0.0em}{3pt}{3pt}
			& $60$ 
			& \tabincell{c}{$\mathcal{\hat{A}}=\{1,2\}:$\\ $f(\mathcal{\hat{A}})=0.3379$}
			&$\checkmark$ \\

			\specialrule{0.05em}{3pt}{3pt}
			\multirow{4}{*}{$Ke^{-t}$}
			& $30$
			& \tabincell{c}{$\mathcal{\hat{A}}=\{1,2\}:$\\ $f(\mathcal{\hat{A}})=4.3e^{-6}$} 
			& $\checkmark$\\
			\specialrule{0.0em}{3pt}{3pt}
			& $60$ 
			& \tabincell{c}{$\mathcal{\hat{A}}=\{1,2\}:$\\ $f(\mathcal{\hat{A}})=8.15e^{-13}$}
			&$\checkmark$ \\

			\specialrule{0.05em}{3pt}{3pt}
			\multirow{4}{*}{$K\mathrm{randn}$}
			& $30$
			& \tabincell{c}{$\mathcal{\hat{A}}=\{1,2\}:$\\ $f(\mathcal{\hat{A}})=0.0400$} 
			& No \\
			\specialrule{0.0em}{3pt}{3pt}
			& $60$ 
			& \tabincell{c}{$\mathcal{\hat{A}}=\{1,2\}:$\\ $f(\mathcal{\hat{A}})=0.0390$}
			& No \\
			
			\specialrule{0.15em}{3pt}{3pt}
		\end{tabular}
	\end{center}
\end{table}

\subsection{Comparison of Algorithms}\label{Algorithms}
The comparison of the running time between Algorithm $1$ and Algorithm $2$ is shown in Fig. \ref{fig7}. It illustrates that Algorithm $2$ significantly reduces the running time when the system suffers from time-invariant attacks. When the attack cost $c(i)=1$ for any $i\in \mathcal{V}$, Fig. \ref{fig7a} shows the attenuation of the running time decreases as the total attack cost increases. Compared with Fig. \ref{fig7b}, we find that Algorithm $2$ has better performance in terms of the running time when the attack cost of each agent $c(i)=d_i$. 
In addition,
when $\Omega=1$, Algorithm $1$ needs to run $810.778127s\approx 0.225$ hours.
Based on the worst-case running time $T_w \sim \mathcal{O}(|\Omega| |\mathcal{V}| n^3 m^3)$, we know the worst-case running time depends on the number of agents, the dimension of the state, and the given total attack cost. When the number of the agents $n=50$ and the dimensions of the state $m=2$, Algorithm $1$ will run about $1085$ hours, which takes a long time. Algorithm $2$ can save $268$ hours, which accounts for $24.7\%$ running time of Algorithm $1$. In large-scale systems, as the number of agents increases, Algorithm $2$ takes less running time than Algorithm $1$.

\section{Conclusion}\label{VII}
In this paper, we investigated the FDI attack design problem where the adversary considers how to select the subset of agents to inject false data subject to limited cost. We showed the submodularity optimization theory is a helpful tool to deal with the problem, which is NP-hard. First, we proved the submodularity of the convergence error under deterministic time-invariant and time-variant attacks, respectively. Then, we developed FDI-ASSA and IFDI-ASSA to obtain the near-optimal compromised subset solution in polynomial time by combining submodularity and the greedy algorithm. Furthermore, we derived the gap between the obtained solution and the optimal one. In addition, we revealed the effects of dimensions of the state on the running time under the proposed algorithms. Extensive simulation results show that the proposed algorithms provide a greater convergence error than the random and degree-based algorithms regardless of the varying total attack cost and different attack strategies.  
Future works include developing the proposed attack selection algorithm for the stealthy FDI attack and taking the lossy network into consideration for better designing highly robust secure algorithms.

\bibliographystyle{IEEEtran}

\textbf{Xiaoyu Luo} (S'19) received B.E. degree in the Department of Automation from Tianjin Universary, Tianjin, China, in 2019.
She is currently pursuing the Ph.D. degree with the Department of Automation, Shanghai Jiao Tong Universary, Shanghai, China. She is a member of Intelligent of Wireless Networking and Cooperative Control group. Her research interests include fault-tolerant control in multi-agent systems, cooperative charging in energy storage system and security of cyber-physical systems.

\textbf{Chengcheng Zhao} received the PhD degree in control science and engineering from Zhejiang University, Hangzhou, China, in 2018. She is currently a research fellow in the Department of Electrical and Computer Engineering, University of Victoria. Her research interests include consensus and distributed optimization, distributed energy management in smart grids, vehicle platoon, and security and privacy in network systems. She received IEEE PESGM 2017 best conference papers award, and one
of her paper was shortlisted in IEEE ICCA 2017 best student paper award finalist. She is a peer reviewer for Automatica, IEEE Transactions on Information Forensics and Security, IEEE Transactions on Industrial Electronics and etc. She was the TPC member for IEEE GLOBECOM 2017, 2018, and IEEE ICC 2018.

\textbf{Chrongrong Fang} received the B.Sc. degree in automation and the Ph.D. degree in control science and engineering from Zhejiang University, Hangzhou, China, in 2015 and 2020, respectively. He is currently an Assistant Professor with the Department of Automation, Shanghai Jiao Tong University, Shanghai, China. His research interests include anomaly detection and diagnosis in cyber-physical systems and cloud networks.

\textbf{Jianping He} (M’15-SM’19) is currently an associate professor in the Department of Automation at Shanghai Jiao Tong University, Shanghai, China. He received the Ph.D. degree in control science and engineering from Zhejiang University, Hangzhou,China, in 2013, and had been a research fellow in the Department of Electrical and Computer Engineering at University of Victoria, Canada, from Dec. 2013 to Mar. 2017. His research interests mainly include the distributed learning, control and optimization,
security and privacy in network systems.Dr. He serves as an Associate Editor for IEEE Open Journal of Vehicular Technology and KSII Trans. Internet and Information Systems. He was also a Guest Editor of IEEE TAC, International Journal of Robust and Nonlinear Control, etc. He was the winner of Outstanding Thesis Award, Chinese Association of Automation, 2015. He received the best paper award from IEEE WCSP’17, the best conference paper award from IEEE PESGM’17, and was a finalist for the best student paper award from IEEE ICCA’17.

\end{document}